\documentclass[12pt]{amsart}
\usepackage{amsmath,amssymb,latexsym}
\usepackage{fullpage}
\usepackage{amscd}
\usepackage{enumitem}
\theoremstyle{plain}
\newtheorem{theorem}{Theorem}

\newtheorem{proposition}{Proposition}
\newtheorem{corollary}{Corollary}
\newtheorem{lemma}{Lemma}

\newtheorem{conjecture}{Conjecture}

\theoremstyle{remark}
\newtheorem{remark}{Remark}

\title{A new way integral for the $n$-fold cover of $SL_2$}
\author{Tianxiang Liu}

\begin{document}
\begin{abstract}
Let $\pi$ be a genuine irreducible cuspidal automorphic representation of the $n$-fold cover $SL_2^{(n)}(\mathbb{A})$ of $SL_2(\mathbb{A})$, $n \geq 3$ odd, and let $\tau$ be the genuine representation of $GL_1^{(n)}(\mathbb{A})$ obtained from a unitary Hecke character of $\mathbb{A}^\times$. We prove that, for a certain choice of data, a new way integral for $SL_2^{(n)}(\mathbb{A})$ represents the partial $L$-function $L^S(ns+\tfrac12,\pi\times\tau)$. The integral is derived from the generalized doubling integral of Cai, Friedberg, Ginzburg and Kaplan, following the idea of Ginzburg and Soudry.
\end{abstract}
\maketitle

\section{Introduction}\label{sec:intro}

Automorphic $L$-functions are among the central objects of number theory. One way to study them is through their integral representations. In the Rankin--Selberg method, a global integral of automorphic forms unfolds to a unique model such as the Whittaker model, and its uniqueness gives the Euler product. Piatetski-Shapiro and Rallis constructed two integrals that do not use Whittaker models. The doubling integral \cite{PS-R1} integrates two cusp forms of a classical group $G$ against an Eisenstein series on a larger group containing $G \times G$. It unfolds to a matrix coefficient and applies to every cuspidal automorphic representation of $G$. The new way integral \cite{PS-R2} integrates a cusp form on a symplectic group against a theta series and an Eisenstein series and unfolds to a non-unique model. Other new way integrals were constructed in \cite{B-F-G, P-S1, P-S2, Y, J-Y}. Cai, Friedberg, Ginzburg and Kaplan \cite{C-F-G-K1} generalized the doubling integral to represent the tensor product $L$-function $L(\pi \times \tau, s)$ of $\pi$ with a cuspidal representation $\tau$ of $GL_m$, by integrating against a Fourier coefficient of an Eisenstein series induced from a generalized Speh representation attached to $\tau$.

There are also works that generalize the above constructions to covering groups. These groups are number theoretical in nature since they are constructed from Hilbert symbols \cite{K-P}. The extension of the generalized doubling integral to covering groups is sketched in \cite{C-F-G-K2} and carried out in full detail for covers of symplectic groups by Kaplan \cite{K}. Bump and Friedberg \cite{B-F} presented a new new-way integral for a covering group. They showed that the Shimura integral of Bump and Hoffstein for the $n$-fold cover of $GL_r$ can be viewed as a new way integral, and proved the required identity for the generating function when $r=2,3$. Ginzburg \cite{G2} proved this identity for arbitrary $r$.

In this paper, we present a different new way integral for a covering group. Our results are the following: let $\pi$ be a genuine irreducible cuspidal automorphic representation of $SL_2^{(n)}(\mathbb{A})$, $n \geq 3$ odd, and let $\tau$ be a genuine automorphic representation of $GL_1^{(n)}(\mathbb{A})$ obtained from a unitary Hecke character of $\mathbb{A}^\times$. We prove that, for a certain choice of data, a new way type integral for $SL_2^{(n)}(\mathbb{A})$ represents the partial $L$-function $L^S(ns+\tfrac12,\pi\times\tau)$. Ginzburg and Soudry \cite{G-S2} derive new way integrals from the generalized doubling integrals on linear groups. We extend this approach to covering groups. Following them, we start with the Eisenstein series defined on $Sp_{4n}^{(n)}(\mathbb{A})$ in the doubling method of Kaplan \cite{K} for the pair $(\pi,\tau)$. We integrate its Fourier coefficient against a cusp form on one copy of $SL_2^{(n)}$ only, and take a Whittaker coefficient of the resulting function on the other copy. By root exchanges and the unfolding of a Fourier--Jacobi coefficient of the Eisenstein series, this integral is equal to a new way type integral, in which the cusp form is integrated against a theta series on the double cover of $SL_2$ and an Eisenstein series on the $2n$-fold cover of $Sp_{2n}$ (Theorem~\ref{Theorem}). By unfolding as in the doubling method, we also show that the function on the other copy is equal to $L^S(ns+\tfrac12,\pi\times\tau)$, divided by a normalizing factor, times a cusp form in an outer conjugate of $\overline{\pi}$ (Theorem~\ref{L:thm}). It follows that, for a certain choice of data, the new way type integral represents $L^S(ns+\tfrac12,\pi\times\tau)$ (Corollary~\ref{L:cor}). We conjecture that, for all decomposable data, the integral is Eulerian and represents the tensor product $L$-function $L^S(ns+\tfrac12,\pi\times\tau)$. Moreover, David Ginzburg has suggested that the approach presented here may be generalized by
replacing $SL_2$ by a higher rank symplectic group $Sp_{2r}$. Though higher rank will require additional
steps, this paper gives a proof of concept. This is another instance of the paradigm, proposed in \cite{F-G}, that constructions in the theory of automorphic forms should generalize to covers.

We now recall the constructions of \cite{C-F-G-K1} and \cite{G-S2} in more detail. In the construction of the generalized doubling integral, the Eisenstein series is induced from a generalized Speh representation $\Delta(\tau, c)$ on the Siegel parabolic subgroup. Let $\pi$ be an irreducible cuspidal automorphic representation of $Sp_{2k}(\mathbb{A})$ and let $\tau$ be an irreducible cuspidal automorphic representation of $GL_m(\mathbb{A})$. The generalized doubling integral for the pair $(\pi,\tau)$ integrates two cusp forms in the space of $\pi$ against a Fourier coefficient of the Eisenstein series $E(f_{\Delta(\tau,2k),s})$ on $Sp_{4km}(\mathbb{A})$, attached to a section $f_{\Delta(\tau,2k),s}$ of the representation induced from $\Delta(\tau,2k)|\det|^s$ on the Siegel parabolic subgroup. The Fourier coefficient is taken along the unipotent radical $U=U_{(2k)^{m-1},4km}$ of the standard parabolic subgroup whose Levi part is $GL_{2k}\times\dots\times GL_{2k}\times Sp_{4k}$, with respect to a character $\psi_U$. The two copies of $Sp_{2k}$ are embedded in $Sp_{4km}$ by a map $t$ which normalizes $U$ and fixes $\psi_U$. As an application, Cai, Friedberg and Kaplan \cite{C-F-K} used the generalized doubling integrals and the converse theorem to establish the weak functorial lift to $GL_N$ of an arbitrary irreducible cuspidal automorphic representation of a split classical group without any genericity assumption.

Ginzburg and Soudry \cite{G-S2} (see also \cite{G-S1}) derived new integrals from the generalized doubling integrals. They integrate the same Fourier coefficient against one cusp form instead. With $[H]$ denoting $H(F)\backslash H(\mathbb{A})$, they define, for a cusp form $\varphi_\pi$ in the space of $\pi$ and $g\in Sp_{2k}(\mathbb{A})$,
\begin{equation*}
\xi(\varphi_\pi,f_{\Delta(\tau,2k),s})(g)=\int\limits_{[Sp_{2k}]}\:\int\limits_{[U]}\varphi_\pi(h)\,E(f_{\Delta(\tau,2k),s})(u\,t(g,h))\,\psi_U^{-1}(u)\,du\,dh.
\end{equation*}
Theorem 1 of \cite{G-S2} states that, for a $K$-finite holomorphic section, $\xi(\varphi_\pi,f_{\Delta(\tau,2k),s})$ is meromorphic in $s$ and takes values in the space of an outer conjugate $\pi^{\iota}$ of $\pi$ by an element of order two. It is Eulerian and represents the same partial $L$-function as the generalized doubling integral, namely $L^S(\pi\times\tau,s+\tfrac12)$, after normalizing the Eisenstein series.

They propose that all known constructions of global integrals which represent the $L$-functions of the doubling method are derived from the function $\xi(\varphi_\pi,f_{\Delta(\tau,2k),s})$. As evidence, they show that a Fourier coefficient of the function $\xi$ recovers the new way integral of \cite{PS-R2}. For $k=m=2$, they consider the integral
\begin{equation}\label{intro:GS model}
\int\limits_{[\text{Mat}_2^0]}\xi(\varphi_\pi,f_{\Delta(\tau,4),s})\left(\begin{pmatrix} I_2&Z\\ &I_2\end{pmatrix}\right)\psi^{-1}(\text{tr}(TZ))\,dZ,
\end{equation}
where $\text{Mat}_2^0$ is the space of the matrices $Z$ with $\left(\begin{smallmatrix} I_2&Z\\ &I_2\end{smallmatrix}\right)\in Sp_4$, and $T=J_2T_0$ for a symmetric matrix $T_0\in GL_2(F)$. The integral \eqref{intro:GS model} expands to
\begin{equation}\label{intro:GS expanded}
\int\limits_{[\text{Mat}_2^0]}\:\int\limits_{[Sp_4]}\:\int\limits_{[U_{4,16}]}
\varphi_\pi(h)\,E(f_{\Delta(\tau,4),s})\left(u\; t\left(\begin{pmatrix} I_2&Z\\ &I_2\end{pmatrix},h\right)\right)
\psi^{-1}_{U_{4,16}}(u)\,\psi^{-1}(\text{tr}(TZ))\,du\,dh\,dZ.
\end{equation}
They prove that, for nontrivial choices of the sections involved, \eqref{intro:GS expanded} is equal to the new way integral
\begin{equation}\label{intro:GS newway}
\int\limits_{[Sp_4]}\:\int\limits_{[U_{2,8}]}\varphi_\pi(h)\,
\theta^{\phi}_{\psi,8}(l_T(v)i_T(1,h))\,E(f^{\prime}_{\Delta(\tau\otimes\chi_T,2),s})(v(1,h))\,dv\,dh.
\end{equation}

We carry the above results from \cite{G-S2} to the $n$-fold cover of $SL_2$. Fix an odd integer $n \geq 3$ and a number field $F$ containing the group $\mu_{2n}$ of $2n$-th roots of unity. Let $SL_2^{(n)}(\mathbb{A})$ denote the $n$-fold cover of $SL_2(\mathbb{A})$. Let $\pi$ and $\tau$ be as above, let $\overline{\pi}$ be the complex conjugate of $\pi$, and let $\varphi^{(n)}$ be a cusp form in the space of $\overline{\pi}$. Let $E^{(n)}(f_s)$ be the Eisenstein series on the $n$-fold cover $Sp^{(n)}_{4n}(\mathbb{A})$ attached to a section $f_s$ of the representation induced from $\Delta(\Theta_{GL^{(n)}_n(\mathbb{A})}, 2)|\det|^s$ on the Siegel parabolic subgroup, where $\Delta(\Theta_{GL^{(n)}_n(\mathbb{A})}, 2)$ denotes the generalized Speh representation attached to two copies of the theta representation $\Theta_{GL^{(n)}_n(\mathbb{A})}$ of the Siegel Levi cover $GL^{(n)}_n(\mathbb{A})$. Both representations are constructed in \cite[\S 3]{K} from $\tau$.

As in \cite{G-S2}, let $\xi(\varphi^{(n)},f_s)$ be the function on $SL_2^{(n)}(\mathbb{A})$ obtained by integrating the Fourier coefficient of $E^{(n)}(f_s)$ along $U_{2^{n-1},4n}$ against the cusp form on one copy of $SL_2^{(n)}$ only (Section~\ref{L:statement}). In Theorem~\ref{L:thm}, we show that for decomposable data unramified outside a finite set of places $S$, $\xi(\varphi^{(n)},f_s)$ is equal to the partial $L$-function $L^S(ns+\tfrac12,\pi\times\tau)/d^S(s)$ times a vector whose components outside $S$ are spherical. Here $d^S(s)$ is a product of partial $L$-functions of $\tau$. The local factors are the unramified doubling integrals computed by Kaplan \cite[Theorem 66]{K}.

The object of study is the Whittaker coefficient of $\xi(\varphi^{(n)},f_s)$, which is the integral
\begin{equation*}
\int\limits_{[SL_2^{(n)}]}\:\int\limits_{[U_{2^{n-1},4n}]}\:\int\limits_{F\backslash\mathbb{A}} \varphi^{(n)}(\tilde{g})\,E^{(n)}(f_s)\left(u\: t\left(\begin{pmatrix} 1&y\\ &1\end{pmatrix}, \tilde{g}\right)\right)\psi^{-1}_{U_{2^{n-1},4n}}(u)\,\psi^{-1}(y)\,dy\,du\,d\tilde{g},
\end{equation*}
the analogue of \eqref{intro:GS expanded}. Here $t$ embeds $SL_2 \times SL_2$ in $Sp_{4n}$, with the cover carried by the second copy, and $\psi_{U_{2^{n-1},4n}}$ is an automorphic character of the unipotent group $U_{2^{n-1},4n}(\mathbb{A})$. All notation is defined in Sections~\ref{sec:prelim} and \ref{sec:main theorem}.

Our main result, Theorem~\ref{Theorem}, states that for suitable nontrivial choices of sections, this integral is equal to the new way type integral
\begin{equation*}
\int\limits_{[SL_2^{(n)}]}\: \int\limits_{[U_{1^{n-1}, 2n}]}\:
\varphi^{(n)}(\tilde{g})\, \theta^{\phi^\prime}_{2}(l^\prime(v)\,j(\tilde{g}))\,
E^{(2n)}(f^\prime_s)(v\,i(\tilde{g}))\,
\psi^{-1}_{\hat{U}_{1^{n-1}, 2n}}(v)\:dv\,d\tilde{g}.
\end{equation*}
In it, the cusp form is integrated against a theta series $\theta^{\phi^\prime}_{2}$ on the Jacobi group $\mathcal{H}_3(\mathbb{A})SL_2^{(2)}(\mathbb{A})$ of $SL_2$, attached to the character $\psi^\prime(t) = \psi(t/2)$, and an Eisenstein series $E^{(2n)}(f^\prime_s)$ on the $2n$-fold cover $Sp^{(2n)}_{2n}(\mathbb{A})$, induced from $\Theta_{GL_n^{(n)}(\mathbb{A})}\gamma_{\psi}|\det|^{s}$ on the Siegel parabolic subgroup, where $\gamma_\psi$ is the Weil factor attached to $\psi$. The maps $i$, $j$ and $l^\prime$ and the character $\psi_{\hat{U}_{1^{n-1}, 2n}}$ are defined in Section~\ref{sec:main theorem}.

As a corollary, for decomposable data unramified outside $S$, the new way type integral of Theorem~\ref{Theorem} is equal to $L^S(ns+\tfrac12,\pi\times\tau)/d^S(s)$ times a Whittaker coefficient of a vector in $\overline{\pi}^{\iota}$. As in \cite[Conjecture 1]{G-S2}, we expect the new way type integral to be Eulerian and to represent $L^S(ns+\tfrac12,\pi\times\tau)$ (Conjecture~\ref{L:conj}).

The proof of Theorem~\ref{Theorem} is divided into two steps. First we apply the root exchange lemma and Ikeda's method as in \cite{I} to obtain a Fourier--Jacobi coefficient of $E^{(n)}(f_s)$ and the theta series in the new way type integral above. This coefficient is then unfolded to an Eisenstein series on $Sp^{(2n)}_{2n}$. In the resulting integral, the cusp form is integrated against a theta series and an Eisenstein series, as in \cite{PS-R2} and \cite{G-S2}. To prove Theorem~\ref{L:thm}, we unfold $\xi(\varphi^{(n)},f_s)$ as in the generalized doubling method and \cite{G-S2}.

The main new difficulty is the following. The global cover $Sp^{(n)}_{4n}(\mathbb{A})$ is defined by a modified $2$-cocycle $\rho$, obtained from the block-compatible local cocycles of Banks--Levy--Sepanski \cite{B-L-S} and the splitting $\eta$ on the maximal compact subgroup. The modified cocycle is not block-compatible. Conjugation identities that are immediate in the linear case require proof for $\rho$. Section~\ref{sec:cocycle} proves the conjugation statement that the derivation needs. A second delicate point is the embedding of two copies of $SL_2^{(n)}$ in $Sp^{(n)}_{4n}(\mathbb{A})$. The integral of Theorem~\ref{Theorem} only requires the embedding of $U \times SL_2^{(n)}$, where $U$ is an upper triangular unipotent subgroup of the first copy. The function $\xi(\varphi^{(n)},f_s)$ requires both copies, and Section~\ref{L:copies} embeds them with the cocycle inverted on the left copy.

We now describe what is conjectural in the background theory of the generalized Speh representations on covers. For a general genuine cuspidal automorphic representation $\tau$ of a cover of $GL_k$, Kaplan's construction of the generalized Speh representations depends on two conjectures on an intertwining operator and on the partial $L$-function $L^S(s, \tau \times \tau^\vee)$ \cite[Conjectures 50, 51]{K}. Both are theorems in the linear case. In the present paper the inducing data is $\Delta(\Theta_{GL_n^{(n)}(\mathbb{A})}, 2)$, built from the theta representation. This is the case $k = 1$ of Kaplan's construction, and in this case he proves both conjectures \cite[Proposition 52]{K}.
Therefore, this paper does not depend on these conjectures.

The paper is organized as follows. In Section~\ref{sec:prelim}, we set up the preliminaries including the local and global covers of the symplectic groups and the modified cocycle $\rho$. Then we state the root exchange lemma, and recall the construction and the main property of the generalized Speh representations $\Delta(\Theta_{GL^{(n)}_n(\mathbb{A})}, 2)$. Section~\ref{sec:main theorem} defines the integral, states the main theorem, and performs the root exchanges, the conjugation, and the application of Ikeda's method following \cite[\S 4.2]{G-S2}. Section~\ref{sec:FJ descent} carries out the unfolding of the Fourier--Jacobi coefficient. Section~\ref{L:sec} defines $\xi(\varphi^{(n)},f_s)$ and proves the $L$-function identity of Theorem~\ref{L:thm}. Then we state Corollary~\ref{L:cor} and Conjecture~\ref{L:conj}. To complete the proof in Section 4, Section~\ref{sec:cocycle} contains the cocycle computations, and Section~\ref{sec:addacolumn} the proof of Lemma~\ref{lem:addacolumn}.

\section{Notations and Preliminaries}\label{sec:prelim}
% General notation about symplectic group and covers
\subsection{General Notations}
Fix an odd integer $n \geq 3$ throughout the paper. Let $F$ be a number field containing the group $\mu_{2n}$ of $2n$-th roots of unity, which we identify with the corresponding subgroup of $\mathbb{C}^\times$ by a fixed embedding. Let $\mathbb{A}$ denote its ring of Adeles. For any positive integer $k$, let $J_k$ denote the matrix with ones on the anti-diagonal and zeros elsewhere. The symplectic group $Sp_{2k}$ can be realized in $GL_{2k}$ by 
\begin{equation*}
    Sp_{2k} = \left\{ g \in GL_{2k} : g^{t} \begin{pmatrix}&J_k\\-J_k&\end{pmatrix} g = \begin{pmatrix}&J_k\\-J_k&\end{pmatrix} \right\}.
\end{equation*}

% conventions of unipotent U: one parameter subgroup
Let $e_{i,j}$ denote the matrix of size $2k$, which has one at the $(i,j)$ entry, and zeros elsewhere. Let 
$$
e_{i,j}'=e_{i,j}-e_{2k-j+1,2k-i+1},\quad  \text{if} \quad 1\le i,j\le k, 
$$
and
$$  
e_{i,j}'=e_{i,j}+e_{2k-j+1,2k-i+1},\quad  \text{if} \quad 1\le i\le k,\quad j>k.
$$

\subsection{Global Covering}
In this subsection, we define the global $n$-fold cover $Sp^{(n)}_{2k}(\mathbb{A})$ of $Sp_{2k}(\mathbb{A})$, following \cite[\S 2.2, \S 2.5]{K}. For each place $\nu$ of $F$, let $\sigma_\nu$ denote the block-compatible $\mu_n$-valued 2-cocycle on the local group $Sp_{2k}(F_\nu)$, obtained by restricting the 2-cocycle on $GL_{2k}(F_\nu)$ constructed in \cite[\S 3]{B-L-S}. The local covering group $Sp^{(n)}_{2k}(F_\nu)$ is realized as the set of pairs $(g,\epsilon)$, where $g \in Sp_{2k}(F_\nu)$, $\epsilon\in\mu_n$, with multiplication $(g,\epsilon)(g^\prime,\epsilon^\prime) = (gg^\prime, \sigma_\nu(g,g^\prime)\epsilon\epsilon^\prime)$.

The following elementary observation will be used repeatedly.
\begin{lemma}\label{lem:divisible}
Let $\nu$ be a finite place with $\nu \nmid n$, and let $G_0$ be a group generated by subgroups isomorphic to $(\mathcal{O}_\nu, +)$. Then every homomorphism $G_0 \rightarrow \mu_n$ is trivial. Therefore, any two splittings over $G_0$ of a $\mu_n$-cover coincide.
\end{lemma}
\begin{proof}
A homomorphism into the abelian group $\mu_n$ is determined by its restrictions to the generating subgroups, and the image of each generating subgroup is an $n$-divisible subgroup of $\mu_n$, hence trivial, since $(\mathcal{O}_\nu, +)$ is $n$-divisible when $\nu \nmid n$. Two splittings over $G_0$ differ by a homomorphism $G_0 \rightarrow \mu_n$.
\end{proof}

For finite $\nu$, let $K_\nu$ denote the hyperspecial maximal compact subgroup $Sp_{2k}(\mathcal{O}_\nu)$, and let $S_0$ denote the finite set of places of $F$ consisting of the archimedean places, the places dividing $2n$, the finite places ramified in $F/\mathbb{Q}$, and the finite places whose residue field has $3$ elements. The product $\prod_\nu \sigma_\nu$ is not defined on $Sp_{2k}(\mathbb{A})\times Sp_{2k}(\mathbb{A})$, since $\sigma_\nu$ is not trivial on $K_\nu \times K_\nu$ for $\nu \notin S_0$. Instead, for $\nu \notin S_0$, the local cover splits over $K_\nu$: there is a map $\eta_\nu : K_\nu \rightarrow \mu_n$ with $\eta_\nu(e) = 1$ such that $y \mapsto (y, \eta_\nu(y))$ is a group homomorphism \cite[\S 2.2]{K}, and this splitting is unique: two splittings differ by a homomorphism $K_\nu \rightarrow \mu_n$, which is trivial because $K_\nu$ is perfect. Extend $\eta_\nu$ to $Sp_{2k}(F_\nu)$ by setting $\eta_\nu(g) = 1$ for $g \notin K_\nu$, and set $\eta_\nu \equiv 1$ for $\nu \in S_0$. Define the modified cocycle
\begin{equation*}
    \rho_\nu(g, g^\prime) = \frac{\eta_\nu(g)\,\eta_\nu(g^\prime)}{\eta_\nu(gg^\prime)}\,\sigma_\nu(g, g^\prime), \qquad g, g^\prime \in Sp_{2k}(F_\nu).
\end{equation*}
Each $\rho_\nu$ is cohomologous to $\sigma_\nu$, and is trivial on $K_\nu \times K_\nu$ for $\nu \notin S_0$. Since any $g, g^\prime \in Sp_{2k}(\mathbb{A})$ satisfy $g_\nu, g^\prime_\nu \in K_\nu$ for almost all $\nu$, the product
\begin{equation*}
    \rho(g, g^\prime) = \prod_\nu \rho_\nu(g_\nu, g^\prime_\nu)
\end{equation*}
is a well-defined $\mu_n$-valued 2-cocycle on $Sp_{2k}(\mathbb{A})$. The global cover $Sp^{(n)}_{2k}(\mathbb{A})$ is the group of pairs $(g,\epsilon)$, where $g \in Sp_{2k}(\mathbb{A})$, $\epsilon \in \mu_n$, with multiplication $(g,\epsilon)(g^\prime,\epsilon^\prime) = (gg^\prime, \rho(g,g^\prime)\epsilon\epsilon^\prime)$.

We record two properties of $\rho$. First, $\rho(u, u^\prime) = 1$ for $u, u^\prime \in N(\mathbb{A})$, where $N$ denotes the upper triangular maximal unipotent subgroup of $Sp_{2k}$, so $u \mapsto (u, 1)$ realizes $N(\mathbb{A})$ as a subgroup of $Sp^{(n)}_{2k}(\mathbb{A})$. In fact, $\sigma_\nu$ is trivial on $N(F_\nu) \times N(F_\nu)$ for every $\nu \notin S_0$ \cite{B-L-S}, so it suffices to show that $\eta_\nu$ is trivial on $N(F_\nu)$. By the definition of $\eta_\nu$, the only case to check is $\nu \notin S_0$ and $u \in N(F_\nu) \cap K_\nu = N(\mathcal{O}_\nu)$. Since $\sigma_\nu$ is trivial on $N(F_\nu)$, both $u \mapsto (u, 1)$ and $u \mapsto (u, \eta_\nu(u))$ are group homomorphisms on $N(\mathcal{O}_\nu)$, hence so is their ratio $\eta_\nu : N(\mathcal{O}_\nu) \rightarrow \mu_n$, which is trivial by Lemma~\ref{lem:divisible}. Second, by the product formula for the $n$-th order Hilbert symbol, $Sp^{(n)}_{2k}(\mathbb{A})$ splits over the rational points: $g \mapsto (g, \eta^{-1}(g))$ is a group homomorphism on $Sp_{2k}(F)$, where $\eta = \prod_\nu \eta_\nu$ \cite[\S 2.5]{K}. We remark that, unlike the local cocycles $\sigma_\nu$, the modified cocycle $\rho$ is not block-compatible. 
% Let $Sp^{(n)}_{2k}(\mathbb{A})$ denote the metaplectic n-fold cover of $Sp_{2k}(\mathbb{A})$. This group consists of pairs $(g, \epsilon)$ where $g\in Sp_{2k}(\mathbb{A})$ and $\epsilon \in \mu_n$. To define this global covering group, we follow \cite{T}. For each place $\nu$ of $F$, let $\sigma_\nu$ denote the block-compatible 2-cocycle on the local group $Sp_{2k}(F_\nu)$ defined by \cite{B-L-S}. Define the adelic 2-cocycle $\sigma$ by 
% \begin{equation*}
%     \sigma (g, g') = \prod\limits_\nu \sigma_\nu(g_\nu, g'_\nu),
% \end{equation*}
% for $g, g' \in Sp_{2k}(\mathbb{A})$. The group structure is defined by
% \begin{equation*}
%     (g,\epsilon)\cdot (g', \epsilon') = (gg', \sigma(g,g')\epsilon \epsilon'),
% \end{equation*}
% for $g, g' \in Sp_{2k}(\mathbb{A})$, $\epsilon, \epsilon' \in \mu_n$.

% Unipotent subgroups and embedding of G times G in H
Since $\rho$ is trivial on $N(\mathbb{A})$, for any upper triangular unipotent subgroup $U$ of $Sp_{2k}$, the map $u \mapsto (u, 1)$ realizes $U(\mathbb{A})$ as a subgroup of $Sp^{(n)}_{2k}(\mathbb{A})$. Since $\mu_n$ is central, $Sp_{2k}(\mathbb{A})$ acts on $Sp^{(n)}_{2k}(\mathbb{A})$ by conjugation: for $x, g \in Sp_{2k}(\mathbb{A})$ and any preimage $\widetilde{x}$ of $x$,
\begin{equation*}
    \widetilde{x}\,(g,\epsilon)\,\widetilde{x}^{-1} = \bigl(xgx^{-1},\ \rho(x,g)\,\rho(xgx^{-1},x)^{-1}\,\epsilon\bigr),
\end{equation*}
independently of the lift $\widetilde{x}$. These two facts will be used implicitly throughout this paper.

For any positive integers $a, b$ such that $ab \leq k$, let $U_{a^b, 2k}$ denote the unipotent radical of the parabolic subgroup of $Sp_{2k}$ whose Levi part is isomorphic to $GL_a\times ...\times GL_a \times Sp_{2(k-ab)}$, where $GL_a$ is repeated $b$ times. We denote by $Q_k$ the Siegel parabolic subgroup of $Sp_{2k}$. We use $[U]$ to denote the automorphic quotient $U(F)\backslash U(\mathbb{A})$. Since $Sp_{2k}^{(n)}(\mathbb{A})$ splits over $Sp_{2k}(F)$, it makes sense to talk about the automorphic quotient $Sp_{2k}(F) \backslash Sp_{2k}^{(n)}(\mathbb{A})$, which we denote by $[Sp_{2k}^{(n)}]$.

% notations of different covers.
Similarly we can define global covering groups $SL_2^{(n)}(\mathbb{A})$, $Sp_{2n}^{(2n)}(\mathbb{A})$, $Sp_{4n}^{(n)}(\mathbb{A})$ and $Sp_{2n}^{(2)}(\mathbb{A})$ and their automorphic quotients. We introduce the conventions of the various covering groups here.  Elements in $SL_2^{(n)}(\mathbb{A})$ are denoted by $(g,\epsilon)$, where $g\in SL_2(\mathbb{A})$, $\epsilon \in \mu_n$.  Elements in $Sp_{2n}^{(2n)}(\mathbb{A})$ are denoted by $\bar{h} = (h, \epsilon)$, where $h\in Sp_{2n}(\mathbb{A})$, $\epsilon\in \mu_{2n}$.

\begin{remark}\label{rem:cover interaction}
We record two facts about how these covering groups interact.

(i) Let $H^\prime$ be a standardly embedded symplectic block subgroup of $Sp_{2k}$, for example $Sp_{2m}$ embedded by $h \mapsto \text{diag}(I_{k-m}, h, I_{k-m})$. Then the restriction of $\rho$ to $H^\prime(\mathbb{A})$ is exactly the modified cocycle defining the $n$-fold cover of $H^\prime(\mathbb{A})$. In fact, at every place, the restriction of $\sigma_\nu$ to $H^\prime(F_\nu)$ is the block-compatible cocycle of $H^\prime(F_\nu)$, by the block-compatibility formula \cite[\S 3, Theorem 11]{B-L-S}, since the determinant factors appearing there are trivial on symplectic blocks. For $\nu \notin S_0$, the restriction of $\eta_\nu$ to $H^\prime(F_\nu) \cap K_\nu = H^\prime(\mathcal{O}_\nu)$ is a splitting of this cocycle, and hence coincides with the splitting used to define the cover of $H^\prime(\mathbb{A})$, because $H^\prime(\mathcal{O}_\nu)$ is perfect \cite[Lemma 11.1]{M} and its splitting is therefore unique. Away from $H^\prime(\mathcal{O}_\nu)$, and at $\nu \in S_0$, both correction maps equal $1$. In particular, embeddings such as $SL_2 \subset Sp_{4n}$ and $Sp_{2n} \subset Sp_{4n}$ lift naturally to embeddings of the corresponding $n$-fold covers.

(ii) Covers of different degrees are compatible through the power maps of $\mu_{2n}$. The Hilbert symbols satisfy $(x,y)_n = (x,y)_{2n}^2$ and $(x,y)_2 = (x,y)_{2n}^n$. Hence the cocycle defining $Sp^{(n)}_{2k}(\mathbb{A})$ (resp. $Sp^{(2)}_{2k}(\mathbb{A})$) is the square (resp. the $n$-th power) of the cocycle defining $Sp^{(2n)}_{2k}(\mathbb{A})$, again using the uniqueness of the splittings over $K_\nu$. Therefore, $(g, \epsilon) \mapsto (g, \epsilon^2)$ and $(g, \epsilon) \mapsto (g, \epsilon^n)$ are group homomorphisms from $Sp^{(2n)}_{2k}(\mathbb{A})$ to $Sp^{(n)}_{2k}(\mathbb{A})$ and $Sp^{(2)}_{2k}(\mathbb{A})$ respectively. Expressions such as $(h, \epsilon^2)$ and $(h, \epsilon^n)$ below are understood through these maps.
\end{remark}

Let $\pi:\mu_{2n}\rightarrow \mu_{n}$ denote the squaring map. Fix a function $s:\mu_{n}\rightarrow \mu_{2n}$ such that $\pi\circ s$ is the identity map on $\mu_n$. Define a function $i:SL_2^{(n)}(\mathbb{A})\rightarrow Sp_{2n}^{(2n)}(\mathbb{A})$ by 
\begin{equation*}
    i((g, \epsilon)) = (\text{diag}(I_{n-1}, g, I_{n-1}), s(\epsilon)).
\end{equation*}
Define another function $j:SL_2^{(n)}(\mathbb{A})\rightarrow Sp_{2n}^{(2)}(\mathbb{A})$ by 
\begin{equation*}
    j((g, \epsilon)) = (\text{diag}(I_{n-1}, g, I_{n-1}), s(\epsilon)^n).
\end{equation*}

\subsection{Root Exchange Lemma}
In this section, we introduce the lemma on global root exchange. The main references are \cite[\S 2]{G-S2} and \cite{G-R-S2}. Let $H$ denote a symplectic group or its cover. Let $C, X, Y$ be $F$-unipotent subgroups of $H$ such that $X$ and $Y$ are abelian, intersect $C$ trivially, normalize $C$, and satisfy $[X, Y] \subset C$. Let $B=CY$, $D=CX$, $A=BX=DY$. Let $\psi_C$ be a character on $C(\mathbb{A})$ that is trivial on $C(F)$ such that 
\begin{enumerate}[label=\arabic*.]
    \item $X(\mathbb{A})$, $Y(\mathbb{A})$ preserve $\psi_C$, when acting by conjugation.
    \item $\psi_C([x,y])$ defines a bi-multiplicative non-degenerate pairing on $X(\mathbb{A})\times Y(\mathbb{A})$.
\end{enumerate}
Extend $\psi_C$ to characters $\psi_B$ of $B(\mathbb{A})$ and $\psi_D$ of $D(\mathbb{A})$ by making $\psi_B$ trivial on $Y(\mathbb{A})$ and $\psi_D$ trivial on $X(\mathbb{A})$.
Let $\mathcal{A}$ be a space of smooth automorphic functions of moderate growth on $H(\mathbb{A})$. Assume that $\mathcal{A}$ is invariant under right translation by elements of $H(\mathbb{A})$, and that, as an $H(\mathbb{A})$ module, it satisfies the Dixmier--Malliavin lemma \cite{D-M}: every $f \in \mathcal{A}$ is a finite sum of convolutions $\int_{X(\mathbb{A})} \xi(x)\, f^{\prime}(\,\cdot\, x)\,dx$ with $f^{\prime} \in \mathcal{A}$ and $\xi \in C_c^{\infty}(X(\mathbb{A}))$. Let $\Omega\subset H(\mathbb{A})$ be the subset of elements $h$ which satisfy the following two properties. 
\begin{enumerate}[label=\arabic*.]
	\item $[Y({\mathbb{A}}),h]\subset D$ and $\psi_D([Y(\mathbb{A}),h])=1$;
	\item For every $x\in X(\mathbb{A})$, $y\in Y(\mathbb{A})$, $[yxy^{-1},h]\in D(\mathbb{A})$ and
	$\psi_D([yxy^{-1},h])=1$.	
\end{enumerate}

\begin{lemma}\label{re lemma}
    For each $f\in \mathcal A$, there exists  $\varphi\in {\mathcal A}$ such that for all $h\in \Omega$,
    \begin{equation}\label{mix11}
    \int\limits_{B(F)\backslash B({\mathbb{A}})}
    f(vh)\psi_B^{-1}(v)dv=
    \int\limits_{D(F)\backslash D({\mathbb{A}})}
    \varphi(uh)\psi^{-1}_D(u)du.
    \end{equation}
\end{lemma}
Explicitly, given $f \in \mathcal{A}$, by the lemma of Dixmier and Malliavin, there exist $f_1,...,f_r\in \mathcal A$, and $\xi_1,...,\xi_r\in C_c^\infty(X({\mathbb{A}}))$, such that for all $z\in H(\mathbb{A})$,
\begin{equation*}
f(z)=\sum_{i=1}^r\int\limits_{X({\mathbb{A}})}\xi_i(x)f_i(zx)dx.
\end{equation*}
We choose $\varphi$ to be 
\begin{equation*}
\varphi(z)=\sum_{i=1}^r\int\limits_{Y({\mathbb{A}})}\hat{\xi}_i(y)f_i(zy)dy,\quad z\in H(\mathbb{A}),
\end{equation*}
where
\begin{equation*}
\hat{\xi}_i(y)=\int\limits_{X({\mathbb{A}})}\xi_i(x)\psi^{-1}_D([x,y])dx.
\end{equation*}
For a proof of the lemma, we refer to the proof of part two of Lemma 1 in \cite{G-S2}.

\subsection{The Generalized Speh Representation on the Siegel Levi Cover}\label{sec:speh}
We recall the construction and the property of the generalized Speh representation appearing in our construction. Let $GL^{(n)}_{d}(\mathbb{A})$ denote the covering group of $GL_d(\mathbb{A})$ obtained by restricting $Sp^{(n)}_{2d}(\mathbb{A})$ to the Levi part of the Siegel parabolic subgroup, via $g \mapsto \text{diag}(g, g^*)$. This cover is denoted by $GL_d^{(n,n)}$ for odd $n$ in \cite[\S 2.7]{K}. Let $\chi$ be a unitary Hecke character of $\mathbb{A}^\times$. The center of $GL_1^{(n)}(\mathbb{A})$ is the preimage of $(\mathbb{A}^\times)^n$, and $(x,\epsilon)\mapsto\epsilon\chi(x)$ is a genuine character of it. Extending this character to a maximal abelian subgroup of $GL_1^{(n)}(\mathbb{A})$ and inducing to $GL_1^{(n)}(\mathbb{A})$ gives a genuine unitary irreducible automorphic representation $\tau$ of $GL_1^{(n)}(\mathbb{A})$. The isomorphism class of this representation depends only on the restriction of $\chi$ to $(\mathbb{A}^\times)^n$ by the Stone--von Neumann theorem. For $c = 1, 2$, let $\mathcal{L}_{\tau,c}$ be the residual representation of $GL^{(n)}_{cn}(\mathbb{A})$ generated by the multi-residues at $\bigl(\tfrac{cn-1}{2n}, \tfrac{cn-3}{2n}, \dots, \tfrac{1-cn}{2n}\bigr)$ of the Eisenstein series induced from $|\det|^{\zeta_1}\tau \otimes \cdots \otimes |\det|^{\zeta_{cn}}\tau$ on the Borel subgroup. For $c = 1$, $\mathcal{L}_{\tau,1}$ is irreducible and generic \cite[\S 3.5]{K}. We denote it by $\Theta_{GL^{(n)}_n(\mathbb{A})}$ and call it the theta representation. For $c = 2$, let $\Delta(\Theta_{GL^{(n)}_n(\mathbb{A})}, 2)$ denote an irreducible $(n,2)$ summand $\mathcal{E}_\tau$ of $\mathcal{L}_{\tau,2}$, which exists by \cite[Theorem 54]{K}. It is the generalized Speh representation attached to two copies of $\Theta_{GL^{(n)}_n(\mathbb{A})}$.

The main property of this representation is that its generic Fourier coefficients along the unipotent orbit $(n^2)$ are not identically zero and it does not support any generic Fourier coefficient on an orbit greater than or not comparable with $(n^2)$. In the linear case these properties are proved in \cite{G, J-L}; for the covering case we refer to Kaplan \cite[\S 3]{K}.

\section{Main Theorem}\label{sec:main theorem}

\subsection{The Integral on $SL_2^{(n)}$}
Let $\pi$ be a unitary genuine irreducible cuspidal automorphic representation of $SL_2^{(n)}(\mathbb{A})$. Its complex conjugate $\overline{\pi}$ is its contragredient, and the same holds for the local components. The representation $\overline{\pi}$ is anti-genuine. Let $\varphi^{(n)}$ be a cusp form in the space of $\overline{\pi}$. Let $f_s$ be a section on $Sp^{(n)}_{4n}$ induced from the generalized Speh representation $\Delta(\Theta_{GL^{(n)}_n(\mathbb{A})}, 2) |\cdot|^s$ on the Siegel parabolic subgroup and let $E^{(n)}(f_s)$ be the associated Eisenstein series. We define an embedding $t$ of $SL_2 \times SL_2$ into $Sp_{4n}$ by 
\begin{equation*}
    t(g_1, g_2) = \text{diag}\left(g_1, ... , g_1, \begin{pmatrix} g_{11}&&g_{12}\\&g_2&\\g_{21}&&g_{22} \end{pmatrix}, g_1^{\ast}, ..., g_1^{\ast}\right) \in Sp_{4n}, \quad g_1 =\begin{pmatrix} g_{11}&g_{12}\\g_{21}&g_{22} \end{pmatrix},
\end{equation*}
where $g_1 \in SL_2$ and $g_1^\ast = J_2\,{}^tg_1^{-1}J_2$ appear $n-1$ times. Let $Q$ be a standard parabolic subgroup of $Sp_{4n}$, whose Levi part $M_Q$ is isomorphic to $GL_2\times...\times GL_2\times Sp_{4}$, where $GL_2$ appears $n-1$ times. The image of $t$ lies in the Levi subgroup $M_Q$, normalizes $U_{2^{n-1}, 4n}$ and stabilizes the character $\psi_{U_{2^{n-1}, 4n}}$ defined below. Note that $t(u, I_2)$ is upper triangular unipotent for $u$ in the upper unipotent subgroup $U$ of $SL_2$. Since $\rho$ is trivial on unipotent subgroups, and since by Remark~\ref{rem:cover interaction}(i) the restriction of $\rho$ to the embedded copy of $SL_2$ defines the cover $SL_2^{(n)}(\mathbb{A})$, this embedding is well defined on $U \times SL^{(n)}_2$. This embedding $t$ is enough to define the integral of interest in this paper. The embedding of two copies of $SL_2^{(n)}$ into $Sp^{(n)}_{4n}(\mathbb{A})$, which the function $\xi(\varphi^{(n)},f_s)$ of Section~\ref{L:sec} requires, is given in Section~\ref{L:copies}. In general, Kaplan gave a careful treatment of such embedding of covering groups and proved that the doubling integrals on covers are well defined \cite[\S 2.5]{K}. 

The integral we are interested in is the analogue, for the covering group $SL_2^{(n)}$, of the integral constructed in \cite[\S 4.2]{G-S2} derived from the doubling method:
\begin{equation}\label{original}
\int\limits_{[SL_2^{(n)}]}\int\limits_{[U_{2^{n-1},4n}]}\int\limits_{F\backslash\mathbb{A}} \varphi^{(n)}(\tilde{g})E^{(n)}(f_s)(u\:t\left(\begin{pmatrix}
        1&y\\0&1
    \end{pmatrix}, \tilde{g}\right))\psi^{-1}_{U_{2^{n-1}, 4n}}(u)\psi^{-1}(y)dydud\tilde{g},
\end{equation}
where $\psi$ is a nontrivial character on $\mathbb{A}$ that is trivial on $F$, and $\psi_{U_{2^{n-1}, 4n}}$ is a certain automorphic character on $U_{2^{n-1}, 4n}(\mathbb{A})$ which is trivial on $U_{2^{n-1}, 4n}(F)$. More concretely, if we write the group $U_{2^{n-1}, 4n}(\mathbb{A})$ as 
\[
U_{2^{n-1},\,4n}(\mathbb{A})
=
\left\{
\begin{pmatrix}
I_2 & X_1 &        &        &        & \cdots &        &        &        \\
    & I_2 & X_2    &        &        &        &        &        &        \\
    &     & \ddots & \ddots &        & \ddots &        &        &        \\
    &     &        & I_2    & Y      &        &        &        &  \vdots \\
    &     &        &        & I_4    & Y^\ast &        &        &        \\
    &     &        &        &        & I_2    & X_{n-2}^\ast &  &        \\
    &     &        &        &        &        & \ddots & \ddots &        \\
    &     &        &        &        &        &        & I_2    & X_1^\ast \\
    &     &        &        &        &        &        &        & I_2
\end{pmatrix} :\quad X_i\in \text{Mat}_{2\times2}(\mathbb{A}), Y\in\text{Mat}_{2\times4}(\mathbb{A})
\right\},
\]
then 
\begin{equation*}
    \psi_{U_{2^{n-1}, 4n}} = \psi\left(\sum\limits_{i=1}^{n-2} \text{Tr}(X_i) + Y_{1,1} + Y_{2,4}\right),
\end{equation*}
where $Y_{i,j}$ denotes the $(i,j)$ entry of $Y$.

\begin{theorem}\label{Theorem}
    Given a Schwartz function $\phi^\prime \in \mathcal{S}(\mathbb{A})$, there are nontrivial choices of sections $f_s \in \text{Ind} \:\Delta(\Theta_{GL^{(n)}_n(\mathbb{A})}, 2) |\cdot|^s$ and $f^\prime_s$ such that the integral \eqref{original} is equal to the new way type integral
\begin{equation}
\int\limits_{[SL_2^{(n)}]} \int\limits_{[U_{1^{n-1}, 2n}]}
\varphi^{(n)}(\tilde{g})\, \theta^{\phi^\prime}_{2}(l^\prime(v)j(\tilde{g}))\,
E^{(2n)}(f^\prime_s)(vi(\tilde{g}))\,
\psi^{-1}_{\hat{U}_{1^{n-1}, 2n}}(v)\:dvd\tilde{g},
\end{equation}
where $E^{(2n)}(f^\prime_s)$ is an Eisenstein series on $Sp^{(2n)}_{2n}$ corresponding to a section $f^\prime_s$ of
\begin{equation*}
\text{Ind}_{Q_n(\mathbb{A})}^{Sp^{(2n)}_{2n}(\mathbb{A})}\, \Theta_{GL_n^{(n)}(\mathbb{A})}\gamma_{\psi}|\det\cdot|^{s},
\end{equation*}
where $\gamma_{\psi}$ is the Weil factor attached to $\psi$ (cf. \cite[\S 2.8]{K}) and $\psi^\prime(t) = \psi(t/2)$. The function $\theta^{\phi^\prime}_{2}$ is the theta series on $\mathcal{H}_3(\mathbb{A})SL_2^{(2)}(\mathbb{A})$ attached to $\phi^\prime$ and to the character $\psi^\prime$, where $\mathcal{H}_3$ is the Heisenberg group of $SL_2$, and $l^\prime$ is a homomorphism from $U_{1^{n-1}, 2n}(\mathbb{A})$ to $\mathcal{H}_3(\mathbb{A})$ defined in Section~\ref{sec:collapse}. The maps $i$ and $j$ are introduced in Section~\ref{sec:prelim}. The character $\psi_{\hat{U}_{1^{n-1}, 2n}}$ is an automorphic character of $\hat{U}_{1^{n-1}, 2n}(\mathbb{A})$ that is specified later.
\end{theorem}

\subsection{Using Root Exchange and Conjugation}
First, we apply root exchanges. Let $X_1$ be the unipotent subgroup of $Sp_{4n}$ generated by root subgroups $I_{4n} + x_{2i-1, 2i} e^\prime_{2i-1,2i}$ and $Y_1$ be the subgroup of $Sp_{4n}$ generated by root subgroups $I_{4n} + x_{2i, 2i+1} e^\prime_{2i,2i+1}$ for $1\leq i \leq n-1$. Let $C_1$ be the subgroup of $U_{2^{n-1}, 4n}$ generated by root subgroups that are not in $Y_1$. Let $D_1 = C_1X_1$, and let $\psi_{D_1}$ denote the character of $D_1(\mathbb{A})$ obtained by restricting $\psi_{U_{2^{n-1}, 4n}}$ to $C_1$ and extending it trivially on $X_1$. By Lemma~\ref{re lemma}, there exists another section $f^\prime_s$ such that \eqref{original} is equal to
\begin{equation}\label{rootexchange1}
    \int\limits_{[SL_2^{(n)}]}\int\limits_{[D_1]}\int\limits_{F\backslash\mathbb{A}} \varphi^{(n)}(\tilde{g})E^{(n)}(f^\prime_s)(u\:t\left(\begin{pmatrix}
        1&y\\0&1
    \end{pmatrix}, \tilde{g}\right))\psi^{-1}_{D_1}(u)\psi^{-1}(y)dydud\tilde{g}.
\end{equation}

Among the root subgroups of the form \(I_{4n}+x_{i,j}e'_{i,j}\), the restriction of \(\psi_{D_1}\) is nontrivial precisely on those corresponding to
\[
(i,i+2),\quad 1\le i\le 2n-3,
\]
together with
\[
(2n-2,\,2n+2).
\]

Let $w_1$ be the permutation matrix of $GL_{2n}$ corresponding to the permutation 
\begin{align*}
2i+1 &\mapsto i+1 \quad \text{for } 1 \le i \le n-1, \\
2i   &\mapsto n+i \quad \text{for } 1 \le i \le n-1\\
1, 2n &\mapsto 1, 2n.
\end{align*}
% 3\mapsto 2, 5\mapsto3, ..., 2n-1\mapsto n, 2\mapsto n+1, 4\mapsto n+2, ..., 2n-2 \mapsto 2n-1.
Let $w = \text{diag}(w_1, w_1^\ast) \in Sp_{4n}(F)$. Note that $w_1^\ast = w_1$ for this particular $w_1$. Since $E^{(n)}(f^\prime_s)$ is left invariant under $Sp_{4n}(F)$, we may replace $u$ by $wu$ in \eqref{rootexchange1}. Denote the right $w^{-1}$-translation of $f^\prime_s$ by $f^{\prime\prime}_s$. The subgroup of all elements $wu\:t\left(\begin{pmatrix}1&y\\0&1\end{pmatrix}, I_2\right)w^{-1}$ is generated by the root subgroups
$ I_{4n} + x_{i,j} e^\prime_{i,j} $
for the following pairs $(i,j)$:
\begin{enumerate}
    \item the pairs $(i,j)$ with $1\leq i<j \leq 2n$ that are not of the form
    $2 \le i \le n$ and $n+1 \le j \le n+i-1$,
    nor $(n,2n)$;

    \item the lower triangular pairs $(i,j)$ with $n+1 \le i \le 2n-2$ and
    $i-(n-2) \le j \le n$;

    \item the pairs $(i,j)$ with $1 \le i \le 2n$ and $2n+1 \le j \le 4n+1-i$,
    except $(n,2n+1)$ and $(2n,2n+1)$.
\end{enumerate}

Denote this subgroup by $B_2$. Then, \eqref{rootexchange1} can be written as
\begin{equation}\label{afterconjugation}
    \int\limits_{[SL_2^{(n)}]}\int\limits_{[B_2]}  \varphi^{(n)}(\tilde{g})E^{(n)}(f^{\prime\prime}_s)(ut(I_2, \tilde{g}))\psi^{-1}_{B_2}(u)dud\tilde{g},
\end{equation}
where $\psi_{B_2}$ is an automorphic character on $B_2(\mathbb{A})$. Among the root subgroups of the form \(I_{4n}+x_{i,j}e'_{i,j}\), the restriction of \(\psi_{B_2}\) is nontrivial precisely on those corresponding to
\[
(i, i+1),\quad 1\leq i\leq 2n-2 \text{ and } i\neq n,
\]
together with
\[
(n,\, 2n+2)\quad \text{and}\quad (n,\, 3n+1).
\]

The group $B_2$ contains root subgroups below the diagonal, which we remove by a second root exchange. Let $X_2$ be the subgroup of $Sp_{4n}$ generated by root subgroups $I_{4n} + x_{i,j} e_{i,j}$ for $n+1 \leq i \leq 2n-2$ and $i-(n-2)\leq j \leq n$. Let $Y_2$ be the subgroup of $Sp_{4n}$ generated by the root subgroups $I_{4n} + x_{i,j} e_{i,j}$ for $n+1 \leq j \leq 2n-2$ and $j-n+1\leq i \leq n-1$. Let $C_2$ be the subgroup of $B_2$ generated by the root subgroups not in $X_2$. Let $D_2 = C_2Y_2$. It is the subgroup of the full upper triangular unipotent subgroup of $Sp_{4n}$ with zeros at the entries $(n,j)$ for $n+1 \leq j \leq 2n+1$, $(2n, 2n+1)$ and $(2n, 3n+1)$. By Lemma~\ref{re lemma}, there exists a section $f^{'''}_s$ such that \eqref{afterconjugation} is equal to 
\begin{equation}\label{rootexchange2}
    \int\limits_{[SL_2^{(n)}]}\int\limits_{[D_2]} \varphi^{(n)}(\tilde{g})E^{(n)}(f^{'''}_s)(ut(I_2, \tilde{g}))\psi^{-1}_{D_2}(u)dud\tilde{g},
\end{equation}
where $\psi_{D_2}$ is an automorphic character on $D_2(\mathbb{A})$. Among the root subgroups of the form \(I_{4n}+x_{i,j}e'_{i,j}\), the restriction of \(\psi_{D_2}\) is nontrivial precisely on those corresponding to
\[
(i, i+1),\quad 1\leq i\leq 2n-2 \text{ and } i\neq n,
\]
together with
\[
(n,\, 2n+2)\quad \text{and}\quad (n,\, 3n+1).
\]

\paragraph{\textbf{More notation.}}
The group $D_2$ can be written as the product $U_0Y_0\:\hat{U}_{1^{n-1}, 2n}$. The group $U_0$ is the subgroup of $U_{1^n, 4n}$ with zeros at the $(n,j)$ entries for $n+1\leq j\leq 3n$. The group $Y_0$ is the subgroup of $U_{1^n, 4n}$ where all strictly upper triangular entries are zero except the $(n,j)$ for $2n+2\leq j\leq 3n$. The group $U_{1^{n-1}, 2n}$ is the subgroup of the upper unipotent subgroup of $Sp_{2n}$ whose $(n, n+1)$ entry is zero. Let $\hat{U}_{1^{n-1}, 2n}$ denote the embedding of $U_{1^{n-1}, 2n}$ in $Sp^{(n)}_{4n}$ given by $g\mapsto \text{diag}(I_n, g, I_n)$. Let $\psi_{U_0}$, $\psi_{Y_0}$, and $\psi_{\hat{U}_{1^{n-1}, 2n}}$ be the corresponding restriction of $\psi_{D_2}$ to $U_0$, $Y_0$ and $\hat{U}_{1^{n-1}, 2n}$ respectively.

For any element $u$ in $U_{1^n, 4n}$, denote the $(i,j)$ entry by $u_{i,j}$. Let 
\begin{align*}
    x & = (u_{n, n+1}, u_{n, n+2}, ..., u_{n,2n})\\
    y & = (u_{n, 2n+1}, u_{n, 2n+2}, ..., u_{n, 3n})\\
    z &= u_{n, 3n+1}.
\end{align*}
Let $\mathcal{H}_{2n+1}(\mathbb{A}) =  X(\mathbb{A})\oplus Y(\mathbb{A}) \oplus z(\mathbb{A})$ denote the Heisenberg group of \cite[\S 1.3]{G-R-S1}, with the product
\begin{equation*}
    (x_1,y_1,z_1)(x_2,y_2,z_2) = \left(x_1+x_2,\ y_1+y_2,\ z_1+z_2+\tfrac12(x_1J_n\,{}^ty_2-y_1J_n\,{}^tx_2)\right).
\end{equation*}
Let $l$ denote the homomorphism from $U_{1^n, 4n}(\mathbb{A})$ to $\mathcal{H}_{2n+1}(\mathbb{A})$ given by $l(u) = (x,y,\tfrac12 z)$. Put $\psi_2(t) = \psi(2t)$. Then $\psi_2$ on the center of $\mathcal{H}_{2n+1}(\mathbb{A})$ pulls back under $l$ to $\psi_{U_0}$ on the root subgroup at the entry $(n,3n+1)$.

Under these notations, the integral becomes 
\begin{equation}\label{expanded}
    \int\limits_{[SL_2^{(n)}]} \int\limits_{[U_{1^{n-1}, 2n}]} 
    \int\limits_{[Y_0]} \varphi^{(n)}(\tilde{g}) \int\limits_{[U_0]}  E^{(n)}(f^{'''}_s) (u_0y_0\hat{v}\:t(I_2, \tilde{g})) \psi^{-1}_{U_0}(u_0)\psi^{-1}_{Y_0}(y_0) \psi^{-1}_{\hat{U}_{1^{n-1}, 2n}} (v) du_0 dy_0dvd\tilde{g}.
\end{equation}

\subsection{Application of Ikeda's method}
% (Our $U_{1^{n}, 4n}$ will be Ikeda's $V$. Our ${Sp}^{(2n)}_{4n}$ will be Ikeda's $H$. The semidirect product of $U_{1^{n}, 4n}$ and ${Sp}^{(2n)}_{4n} \subseteq Sp_{4n}$ will be Ikeda's $D$. The semidirect product is induced by conjugation action of ${Sp}^{(2n)}_{4n}$ on $U_{1^{n}, 4n}$).
In this section, we apply Ikeda's method \cite[\S 1]{I} to the $du_0$ integral in \eqref{expanded} to obtain a Fourier--Jacobi coefficient. Let $\omega_{\psi_2}$ denote the Weil representation of $\mathcal{H}_{2n+1}(\mathbb{A})Sp^{(2)}_{2n}(\mathbb{A})$ on $\mathcal{S}(X(\mathbb{A}))$ attached to the character $\psi_2$, given by the formulas of \cite[\S 1.6]{G-R-S1}. For $\phi\in\mathcal{S}(X(\mathbb{A}))$, let
\begin{equation*}
    \theta^{\phi}(l(u)\tilde{k}) = \sum_{\xi\in X(F)}\omega_{\psi_2}(l(u)\tilde{k})\phi(\xi), \qquad u\in U_{1^n,4n}(\mathbb{A}),\ \tilde{k}\in Sp^{(2)}_{2n}(\mathbb{A}),
\end{equation*}
denote the theta series. We prove the following proposition.

\begin{proposition}\label{ikeda prop}
Let $f_s$ be a smooth holomorphic section and let $\phi_1, \phi_2 \in \mathcal{S}(X(\mathbb{A}))$ be Schwartz functions. Define $f_s^{\phi_1, \phi_2} = \rho(\varphi^{\phi_1, \phi_2})f_s$, where the function $\varphi^{\phi_1, \phi_2}$ and the convolution operator $\rho(\varphi^{\phi_1, \phi_2})$ are given by \eqref{definition of phi} and \eqref{convolved} below. Then the inner integral of \eqref{expanded}
\begin{equation}\label{afterikedaleft}
     \int\limits_{[U_0]}  E^{(n)}(f^{\phi_1, \phi_2}_s) (u_0y_0\hat{v}\:t(I_2, \tilde{g})) \psi^{-1}_{U_0}(u_0) \psi^{-1}_{\hat{U}_{1^{n-1}, 2n}} (v) du_0
\end{equation}
is equal to 
\begin{equation}\label{afterikedaright}
    \theta^{\phi_1}(l(y_0)vj(\tilde{g})) \int\limits_{[U_{1^n,4n}]}\overline{\theta^{\phi_2}(l(u)vj(\tilde{g}))}E^{(n)}(f_s)(u\hat{v}\:t(I_2, \tilde{g}))\psi^{-1}_{U_{1^n,4n}}(u)du.
\end{equation}

\end{proposition}
\begin{proof}
For $y_0\in Y_0(\mathbb{A}), (h, \epsilon) \in Sp^{(2n)}_{2n}(\mathbb{A})$, let $I(y_0,(h, \epsilon))$ denote the integral
\begin{equation*}
    \int\limits_{[U_0]} E^{(n)}(f_s)(u_0y_0(\hat{h}, \epsilon^2))\psi_{U_0}^{-1}(u_0)du_0.
\end{equation*}
For $u^\prime\in U_{1^n, 4n}(\mathbb{A})$, let $x, y, z$ be its entries in row $n$ as above, so that $l(u^\prime) = (x,y,\tfrac12 z)$, and for $\phi_1, \phi_2 \in \mathcal{S}(X(\mathbb{A}))$ define
\begin{equation}\label{definition of phi}
    \varphi^{\phi_1, \phi_2}(u^\prime) = \int\limits_{X(\mathbb{A})}\phi_1(\xi-\tfrac{x}{2})\overline{\phi_2(\xi+\tfrac{x}{2})} \psi(-z-2\xi J_n\,{}^ty)\:d\xi.
\end{equation}

Define the operator $\rho(\varphi^{\phi_1, \phi_2})$ on $I(y_0,(h, \epsilon))$ by 
\begin{equation}\label{convolved}
    \rho(\varphi^{\phi_1, \phi_2})I(y_0,(h,\epsilon)) = \int\limits_{U_0(\mathbb{A})\backslash U_{1^n,4n}(\mathbb{A})} \varphi^{\phi_1, \phi_2}(u)I(y_0,(h, \epsilon)u)du.
\end{equation}
On the one hand, we can identify $U_0(\mathbb{A})\backslash U_{1^n,4n}(\mathbb{A})$ with $(X\oplus Y)(\mathbb{A})$. Let $U_1$ (resp. $U_2$) be the subgroup of $U_{1^n, 4n}$ that isomorphically maps onto $X$ (resp. $Y$) under $l$. Then \eqref{convolved} can be written as 
\begin{equation*}
    \int\limits_{[U_0]}E^{(n)}(f^{\phi_1, \phi_2}_s) (u_0y_0(\hat{h}, \epsilon^2)) \psi^{-1}_{U_0}(u_0) du_0,
\end{equation*}
where 
\begin{equation*}
    f^{\phi_1, \phi_2}_s = \int\limits_{U_1(\mathbb{A})}\int\limits_{U_2(\mathbb{A})}\varphi^{\phi_1, \phi_2}(u_2u_1)\rho(u_2)\rho(u_1)f_s\: du_1 du_2.
\end{equation*}
On the other hand, we can unfold \eqref{convolved} as in Proposition 1.3 in \cite{I}. For $u \in U_{1^n,4n}(\mathbb{A})$, set
\begin{equation*}
    J(u) = \int\limits_{[U_0]} E^{(n)}(f_s)(u_0\,u\,(\hat{h}, \epsilon^2))\psi^{-1}_{U_0}(u_0)du_0,
\end{equation*}
so that $I(y_0,(h,\epsilon)) = J(y_0)$. In \eqref{convolved}, make the change of variable $u \mapsto (\hat{h},\epsilon^2)^{-1}y_0^{-1}u(\hat{h}, \epsilon^2)$. This is a left translation by $y_0^{-1}$ composed with a conjugation, both of which preserve the coset space $U_0(\mathbb{A})\backslash U_{1^n,4n}(\mathbb{A})$ and its measure. This change of variable turns the translated integral $I(y_0, (h,\epsilon)\,\cdot\,)$ into $J(u)$. We obtain
\begin{align*}
    \rho(\varphi^{\phi_1, \phi_2})I(y_0,(h, \epsilon)) &= \int\limits_{U_0(\mathbb{A})\backslash U_{1^n,4n}(\mathbb{A})} \varphi^{\phi_1, \phi_2}((\hat{h},\epsilon^2)^{-1}y_0^{-1}u(\hat{h}, \epsilon^2))J(u)du\\
    &= \int\limits_{U_0(\mathbb{A})U_{1^n,4n}(F)\backslash U_{1^n,4n}(\mathbb{A})} \sum\limits_{\gamma'\in U_0(F)\backslash U_{1^n, 4n}(F)} \varphi^{\phi_1, \phi_2}((\hat{h},\epsilon^2)^{-1}y_0^{-1}\gamma'u(\hat{h}, \epsilon^2))J(\gamma'u)du.
\end{align*}
Here $J(\gamma'u) = J(u)$ since the rational element $\gamma'$ normalizes $U_0$ and preserves $\psi_{U_0}$, and $E^{(n)}(f_s)$ is left invariant under $(\gamma',1)$. Moreover, since the cocycle $\rho$ is trivial on the upper triangular unipotent matrices $u$, $y_0$ and $\gamma'$, the above is equal to
\begin{equation*}
    \rho(\varphi^{\phi_1, \phi_2})I(y_0,(h, \epsilon)) = \int\limits_{U_0(\mathbb{A})U_{1^n,4n}(F)\backslash U_{1^n,4n}(\mathbb{A})} \sum\limits_{\gamma'\in U_0(F)\backslash U_{1^n, 4n}(F)} \varphi^{\phi_1, \phi_2}(\hat{h}^{-1}y_0^{-1}\gamma'u\hat{h})J(u)du.
\end{equation*}

Lemma 1.1 of \cite{I} says that the inner sum over $\gamma'$ is equal to $\theta^{\phi_1}(l(y_0)(h, \epsilon^n))\overline{\theta^{\phi_2}(l(u)(h, \epsilon^n))}$. Therefore, the above is equal to
\begin{align}
    &\rho(\varphi^{\phi_1, \phi_2})I(y_0,(h, \epsilon)) \notag\\ 
    &= \theta^{\phi_1}(l(y_0)(h, \epsilon^n))\int\limits_{U_0(\mathbb{A})U_{1^n,4n}(F)\backslash U_{1^n,4n}(\mathbb{A})}\overline{\theta^{\phi_2}(l(u)(h, \epsilon^n))}\int\limits_{[U_0]} E^{(n)}(f_s)(u_0u(\hat{h}, \epsilon^2))\psi^{-1}_{U_0}(u_0)du_0 du \notag\\
    &=\theta^{\phi_1}(l(y_0)(h, \epsilon^n))\int\limits_{U_0(\mathbb{A})U_{1^n,4n}(F)\backslash U_{1^n,4n}(\mathbb{A})}\int\limits_{[U_0]}\overline{\theta^{\phi_2}(l(u_0)l(u)(h, \epsilon^n))} E^{(n)}(f_s)(u_0u(\hat{h}, \epsilon^2))\psi^{-1}_{U_{1^n, 4n}}(u_0)du_0 du \notag\\
    &=\theta^{\phi_1}(l(y_0)(h, \epsilon^n)) \int\limits_{[U_{1^n,4n}]}\overline{\theta^{\phi_2}(l(u)(h, \epsilon^n))}E^{(n)}(f_s)(u(\hat{h}, \epsilon^2))\psi^{-1}_{U_{1^n,4n}}(u)du \notag
\end{align}
where in the second step, we combined the action of the Weil representation by $l(u_0)$ and the character $\psi^{-1}_{U_0}(u_0)$ into $\psi^{-1}_{U_{1^n, 4n}}(u)$. The proposition follows by replacing $(h,\epsilon)$ with $\hat{v}i(\tilde{g})$.
\end{proof}

\subsection{Choice of Section}
In this section, we specify the choices of sections that make \eqref{original} equal to \eqref{afterikedaleft}. We will repeatedly use Lemma~\ref{re lemma}.

Let $f^{'''}_s$ be a smooth, holomorphic section as in \eqref{rootexchange2}, and let $\phi_1, \phi_2 \in \mathcal{S}(X(\mathbb{A}))$. Set $f^{\phi_1, \phi_2}_s := \rho(\varphi^{\phi_1,\phi_2})f^{'''}_s$, where $\varphi^{\phi_1,\phi_2}$ is given by \eqref{definition of phi}. By the Dixmier--Malliavin lemma, there exist Schwartz functions $\xi_i \in \mathcal{S}(Y_2(\mathbb{A}))$, $1\leq i\leq r$, and smooth, holomorphic sections $f^{(i)}_s$, such that
\begin{equation*}
f^{\phi_1, \phi_2}_s = \sum_{i=1}^{r}\int\limits_{X_2(\mathbb{A})}\hat{\xi}_i(x)\rho(x)f^{(i)}_s\,dx.
\end{equation*}
Define
\begin{equation}\label{beforesecond}
f^{\prime}_s = \sum_{i=1}^{r}\int\limits_{Y_2(\mathbb{A})}\xi_i(y)\rho(y)f^{(i)}_s\,dy,
\end{equation}
and let $h^{(i)}_s$ denote the $i$-th summand. Again by the Dixmier--Malliavin lemma, there exist Schwartz functions $\xi_{i,j} \in \mathcal{S}(X_1(\mathbb{A}))$, $1\leq i\leq r$, $1\leq j\leq r^{\prime}$, and smooth, holomorphic sections $f^{(i,j)}_s$, such that
\begin{equation*}
\rho(w)h^{(i)}_s = \sum_{j=1}^{r^{\prime}}\int\limits_{Y_1(\mathbb{A})}\hat{\xi}_{i,j}(y)\rho(y)f^{(i,j)}_s\,dy.
\end{equation*}
Define
\begin{equation*}
f_s = \sum_{i=1}^{r}\sum_{j=1}^{r^{\prime}}\int\limits_{X_1(\mathbb{A})}\xi_{i,j}(x)\rho(x)f^{(i,j)}_s\,dx.
\end{equation*}
Then, with the section $f_s$, the integral \eqref{original} is equal to \eqref{expanded} with the section $f^{\phi_1, \phi_2}_s$. 

By Proposition~\ref{ikeda prop}, applied with $f^{'''}_s$ in place of $f_s$, the inner integral of \eqref{expanded} with the section $f^{\phi_1, \phi_2}_s$ is equal to \eqref{afterikedaright} with the section $f^{'''}_s$. Therefore, \eqref{original} is equal to

\begin{equation}\label{afterikeda}
\begin{aligned}
&\int\limits_{[SL_2^{(n)}]} \int\limits_{[U_{1^{n-1}, 2n}]}
\int\limits_{[Y_0]} \varphi^{(n)}(\tilde{g}) \theta^{\phi_1}(l(y_0)vj(\tilde{g}))
\int\limits_{[U_{1^n, 4n}]} \overline{\theta^{\phi_2}(l(u)vj(\tilde{g}))}
E^{(n)}(f^{'''}_s)(u\hat{v}\:t(I_2, \tilde{g})) \psi^{-1}_{U_{1^n, 4n}}(u)\,du \\
&\qquad\qquad \cdot
\psi^{-1}_{Y_0}(y_0)\,
\psi^{-1}_{\hat{U}_{1^{n-1}, 2n}}(v)
\,\,dy_0\,dv\,d\tilde{g}.
\end{aligned}
\end{equation}

\subsection{Collapse of the $Y_0$ Integral}\label{sec:collapse}
In this section, we show that the inner $[Y_0]$ integral is in fact a theta function on $SL^{(2)}_2(\mathbb{A})\cdot \mathcal{H}_3(\mathbb{A})$.

Let $\mathcal{H}_3$ denote the Heisenberg group with three variables $(x,y,z)$ with $x,y,z\in\mathbb{A}$. The product is given by
\begin{equation*}
    (x_1,y_1,z_1)(x_2,y_2,z_2) = \left(x_1+x_2,\ y_1+y_2,\ z_1+z_2+\tfrac12(x_1y_2-y_1x_2)\right).
\end{equation*}
The Weil representation of $\mathcal{H}_3(\mathbb{A})SL_2^{(2)}(\mathbb{A})$ on $\mathcal{S}(\mathbb{A})$ is given by the formulas of \cite[\S 1.6]{G-R-S1} with $k=1$ and the character $\psi^\prime$, where $\psi^\prime(t) = \psi(t/2)$. In particular,
\begin{equation*}
    \omega_{\psi^\prime}((x,y,z))\phi^\prime(\xi) = \psi^\prime(z+\tfrac12 xy+\xi y)\,\phi^\prime(\xi+x).
\end{equation*}
For $\phi^\prime\in\mathcal{S}(\mathbb{A})$, let $\theta^{\phi^\prime}_{2}(h\tilde{k}) = \sum_{\xi\in F}\omega_{\psi^\prime}(h\tilde{k})\phi^\prime(\xi)$ denote the theta series, where $h\in\mathcal{H}_3(\mathbb{A})$ and $\tilde{k}\in SL_2^{(2)}(\mathbb{A})$. We will use the relation $\gamma_{\psi_a}(x) = (a,x)_2\,\gamma_\psi(x)$ between the Weil factors of $\psi$ and of $\psi_a(t) = \psi(at)$ \cite[\S 2.8, (2.68)]{K}. Since $(2,x)_2 = (\tfrac12,x)_2$, it gives $\gamma_{\psi_2} = \gamma_{\psi^\prime}$.

Let $H$ be the subgroup of $U_{1^{n-1},2n}$ generated by the root subgroups at the entries $(n-1,n)$, $(n-1,n+1)$ and $(n-1,n+2)$. For $h_1,h_2\in H$ with entries $(x_i,y_i,z_i)$ at these three positions, the entries of $h_1h_2$ are $(x_1+x_2,\ y_1+y_2,\ z_1+z_2+x_1y_2-y_1x_2)$. 

Define a map $l^\prime$ from $U_{1^{n-1},2n} (\mathbb{A})$ to $\mathcal{H}_3(\mathbb{A})$ given by 
\begin{equation*}
    l^\prime(v) = \left(v_{n-1,n},\ v_{n-1,n+1},\ \tfrac12 v_{n-1,n+2}\right)\in\mathcal{H}_3(\mathbb{A}),\qquad v\in U_{1^{n-1},2n}(\mathbb{A}),
\end{equation*}
where $v_{i,j}$ denotes the $(i,j)$ entry of $v$.  It is easy to check that $l^\prime$ is a homomorphism on $U_{1^{n-1},2n}(\mathbb{A})$. The middle $SL_2$, embedded in rows and columns $n$, $n+1$, normalizes $H$, and $l^\prime$ is equivariant for its conjugation action. By Remark~\ref{rem:cover interaction}(i), the restriction of $Sp^{(2)}_{2n}(\mathbb{A})$ to the middle $SL_2$ is $SL_2^{(2)}(\mathbb{A})$, and we regard $j(\tilde{g})$ as an element of this group.

\begin{lemma}\label{lem:collapse}
Let $\phi_1\in\mathcal{S}(X(\mathbb{A}))$, and define $\phi^\prime_1\in\mathcal{S}(\mathbb{A})$ by $\phi^\prime_1(\xi) = \phi_1(0,\dots,0,\tfrac12,\tfrac12\xi)$. Then, for $v\in U_{1^{n-1},2n}(\mathbb{A})$ and $\tilde{g}\in SL_2^{(n)}(\mathbb{A})$,
\begin{equation*}
    \int\limits_{[Y_0]}\theta^{\phi_1}(l(y_0)vj(\tilde{g}))\psi^{-1}_{Y_0}(y_0)\,dy_0 = \theta^{\phi^\prime_1}_{2}(l^\prime(v)j(\tilde{g})).
\end{equation*}
\end{lemma}
\begin{proof}
For $y_0\in Y_0(\mathbb{A})$ we have $l(y_0) = (0,y,0) \in \mathcal{H}_{2n+1} (\mathbb{A})$ with $y = (0,y_2,\dots,y_n)$, where $y_i = (y_0)_{n,2n+i}$, and $\psi_{Y_0}(y_0) = \psi(y_2)$. Put $\Phi = \omega_{\psi_2}(vj(\tilde{g}))\phi_1$. By \cite[\S 1.6]{G-R-S1}, $(0,y,0)$ acts on $\Phi$ by the character $\psi_2(\xi J_n\,{}^ty) = \psi(2\xi_{n-1}y_2+2\xi_{n-2}y_3+\dots+2\xi_1y_n)$, $\xi = (\xi_1,\dots,\xi_n)$. Hence the left hand side is equal to
\begin{equation*}
    \sum_{\xi\in F^n}\Phi(\xi)\int\limits_{(F\backslash\mathbb{A})^{n-1}}\psi\big((2\xi_{n-1}-1)y_2+2\xi_{n-2}y_3+\dots+2\xi_1y_n\big)\,dy_2\cdots dy_n = \sum_{\xi_n\in F}\Phi(\tfrac12e_{n-1}+\xi_ne_n),
\end{equation*}
where $e_{n-1} = (0,\dots,0,1,0)$ and $e_n = (0,\dots,0,1)$ in $F^n$. The coordinate $\xi_n$ is not constrained because the entry $(n,2n+1)$ does not belong to $Y_0$.

Next we compute $\Phi$ on the line $\tfrac12e_{n-1}+\xi_ne_n$. Write $v = \text{diag}(A,A^\ast)\begin{pmatrix}I_n&T\\&I_n\end{pmatrix}$ with $A$ upper triangular unipotent. By \cite[\S 1.6]{G-R-S1} and $\det A = 1$, for any Schwartz function $\Phi^\prime\in\mathcal{S}(X(\mathbb{A}))$,
\begin{equation*}
    (\omega_{\psi_2}(v)\Phi^\prime)(\xi) = \psi_2\big(\tfrac12\,\xi ATJ_n\,{}^t(\xi A)\big)\,\Phi^\prime(\xi A) = \psi\big(\xi ATJ_n\,{}^t(\xi A)\big)\,\Phi^\prime(\xi A).
\end{equation*}
At $\xi = \tfrac12e_{n-1}+\xi_ne_n$, we have $\xi A = \tfrac12e_{n-1}+(\xi_n+\tfrac12v_{n-1,n})e_n$. Using the fact that $v_{n,n+1} = 0$, we get
\begin{equation*}
    (\omega_{\psi_2}(v)\Phi^\prime)(\tfrac12e_{n-1}+\xi_ne_n) = \psi\big(\tfrac14v_{n-1,n+2}+\tfrac14v_{n-1,n}v_{n-1,n+1}+\xi_nv_{n-1,n+1}\big)\,\Phi^\prime\big(\tfrac12e_{n-1}+(\xi_n+\tfrac12v_{n-1,n})e_n\big).
\end{equation*}
Put $\zeta = 2\xi_n$ and $f(\zeta) = \Phi^\prime(\tfrac12e_{n-1}+\tfrac12\zeta e_n)$. Then the right hand side is
\begin{equation*}
    \psi^\prime\big(\tfrac12v_{n-1,n+2}+\tfrac12v_{n-1,n}v_{n-1,n+1}+\zeta v_{n-1,n+1}\big)\,f(\zeta+v_{n-1,n}) = (\omega_{\psi^\prime}(l^\prime(v))f)(\zeta).
\end{equation*}
Finally, the middle $SL_2^{(2)}(\mathbb{A})$ acts through the Weil representation $\omega_{\psi_2}$ on the last variable only, so $(\omega_{\psi_2}(j(\tilde{g}))\phi_1)(\tfrac12e_{n-1}+\xi_ne_n) = (\omega_{\psi_2}(j(\tilde{g}))\phi_1^\flat)(\xi_n)$, where $\phi_1^\flat(\xi_n) = \phi_1(\tfrac12e_{n-1}+\xi_ne_n)$. The dilation $(D\phi)(\zeta) = \phi(\tfrac12\zeta)$ of $\mathcal{S}(\mathbb{A})$ intertwines the restrictions of $\omega_{\psi_2}$ and $\omega_{\psi^\prime}$ to $SL_2^{(2)}(\mathbb{A})$. Since $D\phi_1^\flat = \phi^\prime_1$, we get $(\omega_{\psi_2}(j(\tilde{g}))\phi_1^\flat)(\tfrac12\zeta) = (\omega_{\psi^\prime}(j(\tilde{g}))\phi^\prime_1)(\zeta)$. Substituting $\Phi^\prime$ with $\omega_{\psi_2}(j(\tilde{g}))\phi_1$ and summing over $\zeta\in F$, which is the same as summing over $\xi_n\in F$, gives the lemma.
\end{proof}

By Lemma~\ref{lem:collapse}, \eqref{afterikeda} can be rewritten as
\begin{equation}\label{aftercollapse}
\begin{aligned}
&\int\limits_{[SL_2^{(n)}]} \int\limits_{[U_{1^{n-1}, 2n}]}
\varphi^{(n)}(\tilde{g})\, \theta^{\phi^\prime_1}_{2}(l^\prime(v)j(\tilde{g}))
\int\limits_{[U_{1^n, 4n}]} \overline{\theta^{\phi_2}(l(u)vj(\tilde{g}))}
E^{(n)}(f^{'''}_s)(u\hat{v}\:t(I_2, \tilde{g})) \psi^{-1}_{U_{1^n, 4n}}(u)\,du \\
&\qquad\qquad \cdot
\psi^{-1}_{\hat{U}_{1^{n-1}, 2n}}(v)
\,\,dv\,d\tilde{g},
\end{aligned}
\end{equation}
with $\phi^\prime_1(\xi) = \phi_1(0,\dots,0,\tfrac12,\tfrac12\xi)$. To complete the proof of Theorem~\ref{Theorem}, we need to show that the inner $[U_{1^n, 4n}]$ integral is an Eisenstein series on $Sp^{(2n)}_{2n}(\mathbb{A})$. We will prove it in the next section.

\section{Fourier--Jacobi Descent}\label{sec:FJ descent}
To complete the proof of Theorem~\ref{Theorem}, we need to show that the integral 
\begin{equation}\label{beforedescent}
    \mathcal{E}((h, \epsilon)) = \int\limits_{[U_{1^n, 4n}]}\theta^{\phi}(l(u)(h, \epsilon^n))E^{(n)}(f_s)(u (\hat{h}, \epsilon^2)) \psi^{-1}_{U_{1^n, 4n}}(u) du 
\end{equation}
as a function of $\bar{h} = (h, \epsilon)\in Sp^{(2n)}_{2n}(\mathbb{A})$ is an Eisenstein series corresponding to a section of $\text{Ind}_{Q_n(\mathbb{A})}^{Sp^{(2n)}_{2n}(\mathbb{A})}\, \Theta_{GL_n^{(n)}(\mathbb{A})}\gamma_{\psi}|\det\cdot|^{s}$. Here $\hat{h}$ denotes $\text{diag}(I_n,h,I_n)$, and $\gamma_{\psi}$ is the Weil factor attached to $\psi$, for whose definition and properties we refer to \cite[\S 2.8]{K}.

For $\text{Re}(s)$ large, the series defining $E^{(n)}(f_s)$ converges absolutely, uniformly on compact sets, while the theta series is continuous on the compact quotient $[U_{1^n, 4n}]$. Hence, we may exchange the summation and the integral and obtain
\begin{equation}\label{expandeisen}
   \sum\limits_{\gamma \in Q_{2n}(F) \backslash Sp_{4n}(F)} \: \int\limits_{[U_{1^n, 4n}]}\theta^{\phi}(l(u)(h, \epsilon^n))f_s(\gamma u (\hat{h}, \epsilon^2)) \psi^{-1}_{U_{1^n, 4n}}(u) du.
\end{equation}

Here and in the rest of this section, we identify $Sp_{4n}(F)$ with its image in $Sp^{(n)}_{4n}(\mathbb{A})$ under the splitting $\gamma \mapsto (\gamma, \eta^{-1}(\gamma))$ of the rational points. In particular, rational elements such as $\gamma$ and $w_0$ denote their images under this splitting. 

Consider the double coset space
\begin{equation*}
Q_{2n}(F) \backslash Sp_{4n}(F) / U_{1^n, 4n}(F)\widehat{Sp_{2n}}(F),
\end{equation*}
where $\widehat{Sp_{2n}}(F)$ denotes the embedded copy of $Sp_{2n}$ in $Sp_{4n}$ by $g\mapsto \text{diag}(I_n, g, I_n)$. By the Bruhat decomposition, all the representatives in the double coset space can be written in the form
\begin{equation*}
    \begin{pmatrix}
        \epsilon_1&&\epsilon_2\\&I_{2n}&\\ \epsilon_3&&\epsilon_4
    \end{pmatrix},
\end{equation*}
where $\epsilon_1$ and $\epsilon_4$ are $n\times n$ diagonal matrices with 1 or 0 on the diagonal. $\epsilon_2$ (resp. $\epsilon_3$) are anti-diagonal matrices with 1 (resp. -1) or 0 on the anti-diagonal. %(note to myself: it can be written as a product of the Weyl group elements that exchanges $x_i$ and $y_i$ for any $1\leq i \leq n$.)

\begin{lemma} 
The only representative that contributes to \eqref{expandeisen} is $\gamma = \begin{pmatrix}
    0&&J_n \\ &I_{2n}& \\ -J_n&&0
\end{pmatrix}$.
\end{lemma}

\begin{proof}
For any representative $\gamma$ of the double coset, if there is a root subgroup $\{u(x) : x \in \mathbb{A}\}$ of $U_0(\mathbb{A})$ such that $l(u(x))$ is in the center of $\mathcal{H}_{2n+1}(\mathbb{A})$, $\gamma u(x) \gamma^{-1} \in U_{2n}(\mathbb{A})$ (the unipotent radical of $Q_{2n}$) and $\psi_{U_0}(u(x)) \neq 1$, then $\gamma$ contributes zero to the integral \eqref{expandeisen}. In fact, the change of variables $u \mapsto u(x)u$, with $x$ integrated over $F\backslash\mathbb{A}$, leaves the factor $f_s(\gamma u (\hat{h}, \epsilon^2))$ unchanged. The Weil representation action of $l(u(x))$ and the character $\psi^{-1}_{U_{1^n, 4n}}(u(x))$ combine to $\psi^{-1}_{U_0}(u(x))$. Therefore, the term of \eqref{expandeisen} corresponding to $\gamma$ is multiplied by $\int_{F\backslash\mathbb{A}}\psi^{-1}_{U_0}(u(x))\,dx = 0$.
%(\textbf{To remind myself}: $\psi_{U_0}$ is the restriction of $\psi_{D_2}$ on $U_0$. The relevant entries are $(i,i+1)$ for $1\leq i \leq n-1$ and $(n, 3n+1)$ (the entry that corresponds to the $z$ coordinate of $\mathcal{H}_{2n+1}$)).

Let $\epsilon_1(i)$ denote the $(i,i)$ entry of $\epsilon_1$. We inductively show that $\epsilon_1(i)$ is zero for all $1\leq i \leq n$. If $\epsilon_1(n) = 1$, set $u = I_{4n} + xe^{\prime}_{n,3n+1}$ for some $x$ such that $\psi(x) \neq 1$. Then $\gamma u \gamma^{-1} =u$ is in $U_{2n}(\mathbb{A})$ and $\psi_{U_0}(u) \neq 1$. Therefore, such $\gamma$ contributes zero by the above argument. Now assume $\epsilon_1(m) = 0$ for $1 < m \leq n$. Suppose $\epsilon_1(m-1)=1$. Set $u = I_{4n} + xe^{\prime}_{m-1,m}$. Then $\gamma u \gamma^{-1} = I_{4n} + xe^{\prime}_{m-1,4n-m+1}$ is in $U_{2n}(\mathbb{A})$ and $\psi_{U_0}(u) \neq 1$. By induction, $\epsilon_1 = 0$. Once $\epsilon_1 = 0$, the symplectic condition implies that $\epsilon_2 = J_n$, $\epsilon_3 = -J_n$, and $\epsilon_4 = 0$. The lemma follows.
\end{proof}

Instead of using the above $\gamma$, we will use, for convenience, the following representative of the same double coset:
\begin{equation*}
w_0 = \begin{pmatrix}
    &I_n&&\\&&&I_n\\-I_n&&&\\&&I_n&
\end{pmatrix}.
\end{equation*}

We make a factorization $U_{1^n, 4n} = R\hat{B}U^1$. The group $\hat{B}$ is the subgroup of $U_{1^n, 4n}$ which maps isomorphically onto $0\oplus Y \oplus 0\subset \mathcal{H}_{2n+1}(\mathbb{A})$ under the homomorphism $l$ introduced in Section~\ref{sec:main theorem}. The group $U^1$ is of the form 
\begin{equation*}
    \begin{pmatrix}
        I_n&a&0&Z\\&I_n&0&0\\&&I_n&a^\ast\\&&&I_n
    \end{pmatrix}.
\end{equation*}

Let $z$ denote the $(n,1)$ entry of $Z$, that is, the $(n, 3n+1)$ entry of $u_1$. Define a character on $U^1$ by $\psi_{U^1}(u_1)=\psi(z)$. The group $R$ is the subgroup of $U_{1^n, 4n}$ complementary to $U^1$ and $\hat{B}$; its elements have the form
\begin{equation*}
\begin{pmatrix}
    U&0&b&0\\&I_n &0&b^\ast\\&&I_n&0\\&&&U^\ast
\end{pmatrix},
\end{equation*}
where $U$ is an arbitrary upper triangular unipotent $n \times n$ matrix and $b$ is arbitrary except that its last row is zero. Define a character $\psi_{R}$ on $R$ by $\psi_{R}(r) = \psi(\sum_{i=1}^{n-1} U_{i, i+1})$. 

With this factorization, \eqref{expandeisen} becomes
\begin{equation*}
    \sum\limits_{\gamma \in Q_n(F)\backslash Sp_{2n}(F)} \int\limits_{U^1(\mathbb{A})} \int\limits_{[\hat{B}]} \int\limits_{[R]} \theta^{\phi}(l(r) (0,b,0) l(u_1)(g, \epsilon^n))f_s(w_0 r \hat{b} u_1 \hat{\gamma}(\hat{ g}, \epsilon^2)) \psi^{-1}_{U^1}(u_1) \psi^{-1}_{R}(r) dr db du_1.
\end{equation*}

We make a change of variable $r \mapsto r^\prime =w_0^{-1}rw_0$. Since $r^\prime$ has trivial image under $l$, we obtain
\begin{equation*}
    \sum\limits_{\gamma \in Q_n(F)\backslash Sp_{2n}(F)} \int\limits_{U^1(\mathbb{A})} \int\limits_{[\hat{B}]} \int\limits_{[R^\prime]} \theta^{\phi}((0,b,0) l(u_1)(g, \epsilon^n))f_s(r^\prime w_0 \hat{b} u_1 \hat{\gamma}(\hat{g}, \epsilon^2)) \psi^{-1}_{U^1}(u_1) \psi^{-1}_{R^\prime}(r^\prime) dr^\prime db du_1,
\end{equation*}
where $R^\prime$ denotes $w_0^{-1}Rw_0$ and $\psi_{R'}(r') = \psi_{R}(w_0r'w_0^{-1})$.

Let $\iota$ denote the embedding of the full upper unipotent subgroup $U$ of $GL^{(n)}_{2n}$ into $Sp^{(n)}_{4n}$ given by $(u, 1)\mapsto \left(\begin{pmatrix}
    u&0\\0&u^\ast
\end{pmatrix} , 1 \right)$. Then $r^\prime$ can be written as 
\begin{equation*}
    r^\prime = \iota\begin{pmatrix}
1&&&&&x_{1,n+2}&\cdots&x_{1,2n}\\
&1&&&&x_{2,n+2}&\cdots&x_{2,2n}\\
&&\ddots&&&\vdots&\ddots&\vdots\\
&&&1&&x_{n,n+2}&\cdots&x_{n,2n}\\
&&&&1&x_{n+1,n+2}&\cdots&x_{n+1,2n}\\
&&&&&1&\ddots&\vdots\\
&&&&&&\ddots&x_{2n-1,2n}\\
&&&&&&&1
\end{pmatrix},
\end{equation*}
and $\psi_{R^\prime}(r^\prime) = \psi^{-1}(\sum_{i=n+1}^{2n-1}x_{i, i+1})$.

\begin{lemma} \label{lem:addacolumn}
For $h\in Sp^{(n)}_{4n}(\mathbb{A})$, we have 
\begin{multline}\label{addacolumn}
\int\limits f_s\left( \iota\begin{pmatrix}
1&&&&&x_{1,n+2}&\cdots&x_{1,2n}\\
&1&&&&x_{2,n+2}&\cdots&x_{2,2n}\\
&&\ddots&&&\vdots&\ddots&\vdots\\
&&&1&&x_{n,n+2}&\cdots&x_{n,2n}\\
&&&&1&x_{n+1,n+2}&\cdots&x_{n+1,2n}\\
&&&&&1&\ddots&\vdots\\
&&&&&&\ddots&x_{2n-1,2n}\\
&&&&&&&1
\end{pmatrix}h\right)\psi(\sum\limits_{i=n+1}^{2n-1}x_{i, i+1}) dx \\
= \int\limits f_s\left( \iota \begin{pmatrix}
1&&&&x_{1,n+1}&x_{1,n+2}&\cdots&x_{1,2n}\\
&1&&&x_{2,n+1}&x_{2,n+2}&\cdots&x_{2,2n}\\
&&\ddots&&\vdots&\vdots&\ddots&\vdots\\
&&&1&x_{n,n+1}&x_{n,n+2}&\cdots&x_{n,2n}\\
&&&&1&x_{n+1,n+2}&\cdots&x_{n+1,2n}\\
&&&&&1&\ddots&\vdots\\
&&&&&&\ddots&x_{2n-1,2n}\\
&&&&&&&1
\end{pmatrix}h\right)\psi(\sum\limits_{i=n+1}^{2n-1}x_{i, i+1}) dx.
\end{multline}
In both integrals, every variable $x_{i,j}$ is integrated over $F\backslash\mathbb{A}$. The right hand side carries the extra integration over the new $(n+1)$-st column.
\end{lemma}

We postpone the proof of the lemma to Section~\ref{sec:addacolumn}. Denote the second integral of \eqref{addacolumn} by $f^{\psi^\prime}_s$. We can write the integral as 
\begin{equation*}
    \sum\limits_{\gamma \in Q_n(F)\backslash Sp_{2n}(F)} \int\limits_{U^1(\mathbb{A})} \int\limits_{[\hat{B}]} \theta^{\phi}((0,b,0) l(u_1)(g, \epsilon^n))f^{\psi^\prime}_s(w_0\hat{b}u_1\hat{\gamma}(\hat{g}, \epsilon^2)) \psi^{-1}_{U^1}(u_1) db du_1.
\end{equation*}
Conjugate $\hat{b}$ across $w_0$ and make a change of variable in $x_{1, n+1}, x_{2, n+1},..., x_{n, n+1}$. We obtain
\begin{equation*}
        \sum\limits_{\gamma \in Q_n(F)\backslash Sp_{2n}(F)} \int\limits_{U^1(\mathbb{A})} \int\limits_{[\hat{B}]}\theta^{\phi}((0,b,0) l(u_1)(g, \epsilon^n)) db\:f^{\psi^\prime}_s(w_0u_1 \hat{\gamma}(\hat{g}, \epsilon^2))\psi^{-1}_{U^1}(u_1)du_1.
\end{equation*}
Unfolding the theta series and using that the Weil representation of $l(\hat{b}) = (0,b,0)$ acts on $\phi$ by the character $\psi_2(\xi J_n b^t) = \psi(2\xi J_n b^t)$, we have
\begin{equation*}
        \sum\limits_{\gamma \in Q_n(F)\backslash Sp_{2n}(F)} \int\limits_{U^1(\mathbb{A})}  \sum\limits_{\xi \in X(F)}  \omega_{\psi_2}(l(u_1)(g, \epsilon^n))\phi(\xi) \int\limits_{[\hat{B}]} \psi(2\xi J_n b^t)db\:f^{\psi^\prime}_s(w_0u_1\hat{\gamma}(\hat{g}, \epsilon^2))\psi^{-1}_{U^1}(u_1)du_1.
\end{equation*}
When $\xi\neq0$, the inner integral over $[\hat{B}]$ vanishes. In fact, it is an integral over the compact additive group $\hat{B}(F)\backslash\hat{B}(\mathbb{A})$ of a nontrivial character. Therefore, the only term that remains in the sum over $\xi$ is $\xi=0$. The above is equal to
\begin{equation*}
        \sum\limits_{\gamma \in Q_n(F)\backslash Sp_{2n}(F)} \int\limits_{U^1(\mathbb{A})}  \omega_{\psi_2}(l(u_1)(g, \epsilon^n))\phi(0)f^{\psi^\prime}_s(w_0u_1\hat{\gamma}(\hat{g}, \epsilon^2))\psi^{-1}_{U^1}(u_1)du_1.
\end{equation*}
We claim that the section 
\begin{equation*}
    \lambda(f_s, \phi)((g, \epsilon)) = \int\limits_{U^1(\mathbb{A})}  \omega_{\psi_2}(l(u_1)(g, \epsilon^n))\phi(0)f^{\psi^\prime}_s(w_0u_1(\widehat{g}, \epsilon^2))\psi^{-1}_{U^1}(u_1)du_1
\end{equation*}
is a section on $Sp^{(2n)}_{2n}(\mathbb{A})$ induced from the Siegel parabolic subgroup of $Sp^{(2n)}_{2n}(\mathbb{A})$.

To this end, we need the following fact about conjugation in $Sp^{(n)}_{4n}(\mathbb{A})$. Its proof is a computation with the cocycle $\rho$, which we postpone to Section~\ref{sec:cocycle}. Let $\text{diag}(X_1, X_2, X_3, X_4)$ denote the block diagonal matrix with blocks $X_1, \dots, X_4 \in GL_n$. For $A \in GL_n(\mathbb{A})$, set
\begin{equation*}
    m = \text{diag}(I_n, A, A^\ast, I_n), \qquad m^\prime = \text{diag}(A, I_n, I_n, A^\ast),
\end{equation*}
and note that $w_0 m w_0^{-1} = m^\prime$.

\begin{proposition}\label{prop:cocycle global}
Let $A \in GL_n(\mathbb{A})$. For every preimage $\widetilde{w}_0$ of $w_0$ in $Sp^{(n)}_{4n}(\mathbb{A})$ and every $\epsilon \in \mu_n$,
\begin{equation*}
    \widetilde{w}_0\, (m, \epsilon)\, \widetilde{w}_0^{-1} = (m^\prime, \epsilon).
\end{equation*}
\end{proposition}

Consider $g_0 = \left(\begin{pmatrix}A&B\\&A^\ast\end{pmatrix}, \epsilon\right) \in Q_n(\mathbb{A})$, $A \in GL_n(\mathbb{A})$. It is easy to check that the integral is left invariant under $\begin{pmatrix}I_n&B\\&I_n\end{pmatrix}$. Next plug $g_0 = (\text{diag}(A, A^\ast), \epsilon)$ and conjugate it to the left. We obtain $|\det A|^{-n}$ from the change of variables in $u_1$. Then we obtain a factor of $\epsilon^n\gamma_\psi(\det A)|\det A|^{1/2}$ from the action of the Weil representation at $\phi(0)$, where $\gamma_\psi$ is the Weil factor attached to $\psi$. We get another factor of $|\det A|^s \delta^{1/2}_{Q_{2n}}(\text{diag}(A, I_n, I_n, A^\ast)) = |\det A|^{s+(2n+1)/2}$ from the section $f_s$. Finally, notice that the integral $f^{\psi^\prime}_s$ contains the constant term along the standard parabolic subgroup $L$ of $GL_{2n}$ whose Levi part is isomorphic to $GL_n \times GL_n$. This constant term contributes the factor $\Theta_{GL_n^{(n)}(\mathbb{A})}(A, \epsilon^2)\, |\det A|^{-1/2}\, \delta_L^{1/2}(\text{diag}(A,I_n)) = \Theta_{GL_n^{(n)}(\mathbb{A})}(A, \epsilon^2)\, |\det A|^{(n-1)/2}$. Combining all these factors, we get that $g_0$ is transformed as $\Theta_{GL_n^{(n)}(\mathbb{A})}\gamma_\psi|\det\cdot|^s\delta_{Q_n}^{1/2}$. Therefore, $\lambda(f_s, \phi)$ is a section of $\text{Ind}_{Q_n(\mathbb{A})}^{Sp^{(2n)}_{2n}(\mathbb{A})}\, \Theta_{GL_n^{(n)}(\mathbb{A})}\gamma_{\psi}|\det \cdot|^{s}$. It transforms under the center $\mu_{2n}$ of $Sp^{(2n)}_{2n}(\mathbb{A})$ by $\zeta \mapsto \zeta^{n+2}$, a primitive character since $n$ is odd. This completes the proof of Theorem~\ref{Theorem}.

\section{The $L$-function Represented by the Integral}\label{L:sec}

In this section we show that the integral \eqref{original} carries the partial $L$-function $L^S(ns+\tfrac12,\pi\times\tau)$, where $\tau$ is the representation of $GL^{(n)}_1(\mathbb{A})$ fixed in Section~\ref{sec:speh}. We follow \cite[\S 2.3]{G-S2}. In the doubling integral of covers of the symplectic group \cite{K}, the Fourier coefficient of the Eisenstein series along $U_{2^{n-1},4n}$ is integrated against two cusp forms. In this paper and in \cite{G-S2}, we integrate this kernel function against the cusp form on one copy only. The result is a function $\xi(\varphi^{(n)},f_s)$ on the other copy, and the integral \eqref{original} is a Whittaker coefficient of this function. We prove that $\xi(\varphi^{(n)},f_s)$ is a vector in the space of a cuspidal representation. We then show that outside a finite set of places $S$, this vector is $L^S(ns+\tfrac12,\pi\times\tau)/d^S(s)$ times a vector whose components outside $S$ are the spherical vectors, where $d^S(s)$ is a product of partial $L$-functions of $\tau$. The factor at each place outside $S$ is the local doubling integral, which is computed in \cite[Theorem 66]{K}. Hence a linear functional applied to $\xi(\varphi^{(n)},f_s)$ inherits the partial $L$-function.

Throughout this section $\pi$ is a genuine irreducible cuspidal automorphic representation of $SL_2^{(n)}(\mathbb{A})$. Let $\overline{\pi}$ be its complex conjugate, which is anti-genuine. 

\subsection{The Two Copies of $SL_2^{(n)}$}\label{L:copies}
In this section $G=SL_2$ and $H=Sp_{4n}$. Let $P=Q_{2n}$ be the Siegel parabolic subgroup of $H$ and let $U_P$ be its unipotent radical. For $a\in GL_{2n}$ put $a^\ast=J_{2n}\,{}^ta^{-1}J_{2n}$ and $a^{\diamondsuit}=\text{diag}(a,a^\ast)$. Let $U=U_{2^{n-1},4n}$ and $\psi_U=\psi_{U_{2^{n-1},4n}}$ be as in Section~\ref{sec:main theorem}. Let $N_2$ and $N_2^-$ be the upper and the lower triangular unipotent subgroups of $G$, and put $x(y)=\left(\begin{smallmatrix}1&y\\&1\end{smallmatrix}\right)$. For $\nu\notin S_0$ let $K_{G,\nu}=SL_2(\mathcal{O}_\nu)$ and $K_{H,\nu}=Sp_{4n}(\mathcal{O}_\nu)$. We write $\rho$, $\eta_\nu$ and $\eta$ for the cocycle and the maps of Section~\ref{sec:prelim} attached to $H$, and $\rho_2$, $\eta_{2,\nu}$ and $\eta_2$ for those attached to $G$. We write $\sigma_{2,\nu}$ and $\sigma_{4n,\nu}$ for the block-compatible local cocycles $\sigma_\nu$ of Section~\ref{sec:prelim} on $G(F_\nu)$ and $H(F_\nu)$. For a rational element $\gamma$ of $H(F)$ or of $G(F)$ we write $\tilde\gamma$ for its image under the splitting of the rational points, so that $\tilde\gamma=(\gamma,\eta^{-1}(\gamma))$ for $\gamma\in H(F)$.

Put $\mathfrak{e}_1(g)=t(g,I_2)$ and $\mathfrak{e}_2(g)=t(I_2,g)$ for $g\in G$. The subgroups $\mathfrak{e}_1(G)$ and $\mathfrak{e}_2(G)$ commute with each other, normalize $U$, and fix $\psi_U$. We call them the left copy and the right copy of $G$.

\begin{lemma}\label{L:lem cocycle}
For all $g,g^\prime\in G(\mathbb{A})$ we have
\begin{equation}\label{L:eq cocycle}
\rho(\mathfrak{e}_2(g),\mathfrak{e}_2(g^\prime))=\rho_2(g,g^\prime),\qquad
\rho(\mathfrak{e}_1(g),\mathfrak{e}_1(g^\prime))=\rho_2(g,g^\prime)^{-1},
\end{equation}
and the elements $(\mathfrak{e}_1(g),1)$ and $(\mathfrak{e}_2(g^\prime),1)$ of $H^{(n)}(\mathbb{A})$ commute.
\end{lemma}
\begin{proof}
The first cocycle identity follows from block-compatibility. For the second one, fix a place $\nu$. In the notation of \cite{K} we have $m=r=n$ and $k=1$. Then \cite[Proposition 7, (2.22)]{K} gives $\sigma_{4n,\nu}(\mathfrak{e}_1(g),\mathfrak{e}_1(g^\prime))=\sigma_{2,\nu}(g^\ast,g^{\prime\ast})^{-1}$. In the formula of Kubota \cite[(2.13)]{K}, $\sigma_{2,\nu}(g,g^\prime)$ is a Hilbert symbol of ratios of entries of $g$, $g^\prime$ and $gg^\prime$, and these ratios change at most by a sign when we pass to $g^\ast$ and $g^{\prime\ast}$. Since $n$ is odd, $(-1,a)_n=1$ for all $a$, we get $\sigma_{4n,\nu}(\mathfrak{e}_1(g),\mathfrak{e}_1(g^\prime))=\sigma_{2,\nu}(g,g^\prime)^{-1}$. As in Remark~\ref{rem:cover interaction}(i), the uniqueness of the splitting of $K_{G,\nu}$ now gives $\eta_\nu(\mathfrak{e}_1(g))=\eta_{2,\nu}(g)^{-1}$ for all $g\in G(F_\nu)$, and the second identity follows from the definition of $\rho_\nu$.

For fixed $g^\prime$, the commutator of $(\mathfrak{e}_1(g),1)$ and $(\mathfrak{e}_2(g^\prime),1)$ lies in $\mu_n$, and it is a homomorphism in $g$ because $\mu_n$ is central. Every homomorphism from $G(\mathbb{A})$ to $\mu_n$ is trivial. In fact, $G(\mathbb{A})$ is generated by $N_2(\mathbb{A})$ and $N_2^-(\mathbb{A})$, since a bounded number of elementary matrices suffices at every place, and these two groups are divisible. The same holds for $G(F_\nu)$ at every place.
\end{proof}

By Lemma~\ref{L:lem cocycle} the maps
\begin{equation*}
\lambda_R((g,\epsilon))=(\mathfrak{e}_2(g),\epsilon),\qquad\lambda_L((g,\epsilon))=(\mathfrak{e}_1(g),\epsilon^{-1})
\end{equation*}
are injective homomorphisms from $SL_2^{(n)}(\mathbb{A})$ to $H^{(n)}(\mathbb{A})$, and their images commute. At every place we have $\eta_\nu\circ\mathfrak{e}_2=\eta_{2,\nu}$, and $\eta_\nu\circ\mathfrak{e}_1=\eta_{2,\nu}^{-1}$ by the proof of the lemma. So both maps take the image of $G(F)$ in $SL_2^{(n)}(\mathbb{A})$ into the image of $H(F)$ in $H^{(n)}(\mathbb{A})$. On $N_2(\mathbb{A})$ we have $\lambda_L((u,1))=(\mathfrak{e}_1(u),1)$.

We write the involution ${}^{\iota}g=J_2gJ_2$ for $g\in G$. The involution lifts in a unique way to an involution of $SL_2^{(n)}(\mathbb{A})$ which is the identity on the center $\mu_n$, and the same holds for every local cover \cite[Proposition 18]{K}. We denote the lifts by $\tilde g\mapsto{}^{\iota}\tilde g$. The global lift preserves the image of $G(F)$ \cite[(2.49)]{K}. For $\nu\notin S_0$ the local lift preserves the image of $K_{G,\nu}$ under its splitting. For a function $\Phi$ on $SL_2^{(n)}(\mathbb{A})$ put ${}^{\iota}\Phi(\tilde g)=\Phi({}^{\iota}\tilde g)$. For a representation $\sigma$ of $SL_2^{(n)}(\mathbb{A})$ or of a local cover, let $\sigma^{\iota}$ be the representation $\tilde g\mapsto\sigma({}^{\iota}\tilde g)$ on the space of $\sigma$. If $\sigma$ acts on a space of automorphic forms by right translation, then $\Phi\mapsto{}^{\iota}\Phi$ realizes $\sigma^{\iota}$ on the space of the functions ${}^{\iota}\Phi$.

For $g\in G(\mathbb{A})$ choose a preimage $\tilde g$ in $SL_2^{(n)}(\mathbb{A})$ and put
\begin{equation}\label{L:eq diag}
\mathfrak{s}^{\iota}(g)=\lambda_L(\tilde g)\,\lambda_R({}^{\iota}\tilde g).
\end{equation}
This element does not depend on the choice of $\tilde g$, because $\lambda_L$ inverts $\mu_n$ and $\lambda_R$ preserves $\mu_n$. The map $\mathfrak{s}^{\iota}$ is a homomorphism, because $\lambda_L$ and $\lambda_R$ are homomorphisms with commuting images. So $\mathfrak{s}^{\iota}$ is a splitting of the subgroup of the elements $t(g,{}^{\iota}g)$ of $H(\mathbb{A})$. The element $\mathfrak{s}^\iota(g)$ will play the role of $t(b, b^\iota)$ in (2.7) of \cite{G-S2}. The same construction gives the unique local splitting $\mathfrak{s}^{\iota}_\nu$ at every place.

\subsection{The Function $\xi(\varphi^{(n)},f_s)$ and the Main Statement}\label{L:statement}
Let $f_s$ be a smooth holomorphic section of $\text{Ind}_{P(\mathbb{A})}^{H^{(n)}(\mathbb{A})}\Delta(\Theta_{GL^{(n)}_n(\mathbb{A})},2)|\det\cdot|^s$. For $\tilde g\in SL_2^{(n)}(\mathbb{A})$ define
\begin{equation}\label{L:eq xi}
\xi(\varphi^{(n)},f_s)(\tilde g)=\int\limits_{[SL_2^{(n)}]}\int\limits_{[U]}\varphi^{(n)}(\tilde h)\,E^{(n)}(f_s)\big((u,1)\,\lambda_L(\tilde g)\,\lambda_R(\tilde h)\big)\,\psi_U^{-1}(u)\,du\,d\tilde h.
\end{equation}
This is the analogue of the function of \cite[(1.4)]{G-S2} on $SL_2^{(n)}$. The function $\xi(\varphi^{(n)},f_s)$ is anti-genuine, because $\lambda_L$ inverts the center $\mu_n$ and the Eisenstein series is genuine. Since $\lambda_L((x(y),1))=(\mathfrak{e}_1(x(y)),1)$, the integral \eqref{original} is equal to
\begin{equation}\label{L:eq original}
\int\limits_{F\backslash\mathbb{A}}\xi(\varphi^{(n)},f_s)\big((x(y),1)\big)\,\psi^{-1}(y)\,dy.
\end{equation}

Fix an isomorphism $\ell$ from the restricted tensor product $\otimes^\prime_\nu\overline{\pi}_\nu$ to $\overline{\pi}$, and write ${}^{\iota}\ell(v)={}^{\iota}(\ell(v))$. Let $S$ be a finite set of places which contains $S_0$ and the places where $\pi$, $\tau$ or $\psi$ is ramified. Then for $\nu\notin S$ the local component of $\Delta(\Theta_{GL^{(n)}_n(\mathbb{A})},2)$ is the unique irreducible unramified subrepresentation of $\text{Ind}_{\tilde B_{2n}}^{GL^{(n)}_{2n}(F_\nu)}\big((\tau_\nu\otimes\dots\otimes\tau_\nu)\delta_{B_{2n}}^{-1/(2n)}\big)$, and it has a unique functional of the form \eqref{L:eq fW} \cite[Corollary 55 and \S 3.5]{K}. For $\nu\notin S$ let $\varphi^0_\nu$ be a spherical vector of $\overline{\pi}_\nu$. We say that $\varphi^{(n)}$ and $f_s$ are decomposable and unramified outside $S$ if $\varphi^{(n)}=\ell(\otimes_\nu\varphi_\nu)$ with $\varphi_\nu=\varphi^0_\nu$ for $\nu\notin S$, and $f_s=\otimes_\nu f_{s,\nu}$ with $f_{s,\nu}$ the normalized spherical section for $\nu\notin S$. We then write $\varphi_S=\otimes_{\nu\in S}\varphi_\nu$, $\varphi^S=\otimes_{\nu\notin S}\varphi^0_\nu$ and $f_{s,S}=\otimes_{\nu\in S}f_{s,\nu}$. 

Let $\nu\notin S$. By \cite[\S 2.8]{K}, the representation $\pi_\nu$ is the unramified constituent of a principal series representation attached to an unramified quasi-character $\mu$ of $F_\nu^\times$, and its Satake parameter is $t_{\pi_\nu}=\text{diag}(t_\nu,1,t_\nu^{-1})\in SO_3(\mathbb{C})$ with $t_\nu=\mu(\varpi_\nu^n)$. The representation $\tau_\nu$ has a Satake parameter $t_{\tau_\nu}\in\mathbb{C}^\times$. Put
\begin{gather*}
L(s,\pi_\nu\times\tau_\nu)=\det\big(1-t_{\tau_\nu}t_{\pi_\nu}q_\nu^{-s}\big)^{-1},\qquad L(s,\tau_\nu)=\big(1-t_{\tau_\nu}q_\nu^{-s}\big)^{-1},\\
L(s,\tau_\nu,\text{Sym}^2)=\big(1-t_{\tau_\nu}^2q_\nu^{-s}\big)^{-1},\qquad
d_\nu(s)=L(ns+n+\tfrac12,\tau_\nu)\prod_{j=1}^{n}L(2ns+2j-1,\tau_\nu,\text{Sym}^2).
\end{gather*}
Let $L^S(s,\pi\times\tau)$ and $d^S(s)$ be the products over $\nu\notin S$.

\begin{theorem}\label{L:thm}
Let $\varphi^{(n)}$ and $f_s$ be decomposable and unramified outside $S$. For $\text{Re}(s)$ large,
\begin{equation}\label{L:eq euler}
\xi(\varphi^{(n)},f_s)=\frac{L^S(ns+\tfrac12,\pi\times\tau)}{d^S(s)}\ {}^{\iota}\ell\big(\xi_S(\varphi_S,f_{s,S})\otimes\varphi^S\big),
\end{equation}
where $\xi_S(\varphi_S,f_{s,S})$ is a vector in $\otimes_{\nu\in S}\overline{\pi}_\nu$, given by the integral \eqref{L:eq XiS} over $G(F_S)$ in the proof. In particular $\xi(\varphi^{(n)},f_s)$ lies in the space of $\overline{\pi}^{\iota}$.
\end{theorem}

\subsection{Proof of Theorem~\ref{L:thm}}\label{L:proof}
We follow the idea of \cite[\S 2.3]{G-S2}. Let $V$ be the unipotent radical of the standard parabolic subgroup of $GL_{2n}$ with Levi part $GL_2\times\dots\times GL_2$, and for $v\in V$ with blocks $v_{i,j}\in\text{Mat}_2$ put $\psi_V(v)=\psi(\sum_{i=1}^{n-1}\text{Tr}(v_{i,i+1}))$ as in \cite[(3.1)]{K}. Define
\begin{equation}\label{L:eq fW}
f_{\mathcal{W}}(h)=\int\limits_{[V]}f_s\big((v^{\diamondsuit},1)h\big)\,\psi_V(v)\,dv,\qquad h\in H^{(n)}(\mathbb{A}).
\end{equation}
Let
\begin{equation*}
\delta=\begin{pmatrix}&I_{2n}\\-I_{2n}&\end{pmatrix}\begin{pmatrix}I_{2n-2}&&&\\&I_2&I_2&\\&&I_2&\\&&&I_{2n-2}\end{pmatrix},\qquad U_0=U\cap U_P.
\end{equation*}
By \cite[(4.5), (4.6)]{K}, $U=(\delta^{-1}V^{\diamondsuit}\delta)\ltimes U_0$ and $\psi_U(\delta^{-1}v^{\diamondsuit}\delta)=\psi_V^{-1}(v)$.

The proof rests on two facts. First, for $g\in G(\mathbb{A})$ and $h\in H^{(n)}(\mathbb{A})$,
\begin{equation}\label{L:eq claim}
\int\limits_{U_0(\mathbb{A})}f_{\mathcal{W}}\big(\tilde\delta\,(u,1)\,\mathfrak{s}^{\iota}(g)\,h\big)\,\psi_U^{-1}(u)\,du=\int\limits_{U_0(\mathbb{A})}f_{\mathcal{W}}\big(\tilde\delta\,(u,1)\,h\big)\,\psi_U^{-1}(u)\,du.
\end{equation}
This is the analogue of equation (2.7) of \cite{G-S2}. To see this, put $t=t(g,{}^{\iota}g)$. Then $t$ normalizes $U$ and fixes $\psi_U$, and $\delta t\delta^{-1}\in P$, so $t$ normalizes $\delta^{-1}V^{\diamondsuit}\delta$. For $u\in U_0(\mathbb{A})$ write $t^{-1}ut=(\delta^{-1}v^{\diamondsuit}\delta)\,u_0$ with $v\in V(\mathbb{A})$ and $u_0\in U_0(\mathbb{A})$. Conjugation preserves the splittings $u\mapsto(u,1)$ of the unipotent subgroups, so
\begin{equation*}
\tilde\delta\,(u,1)\,\mathfrak{s}^{\iota}(g)=\big(\tilde\delta\,\mathfrak{s}^{\iota}(g)\,\tilde\delta^{-1}\big)\,(v^{\diamondsuit},1)\,\tilde\delta\,(u_0,1).
\end{equation*}
By \cite[Claim 65]{K}, $f_{\mathcal{W}}(\tilde\delta\,\mathfrak{s}^{\iota}(g)\,\tilde\delta^{-1}h)=f_{\mathcal{W}}(h)$, and by \eqref{L:eq fW}, $f_{\mathcal{W}}((v^{\diamondsuit},1)h)=\psi_V^{-1}(v)f_{\mathcal{W}}(h)$. Also $\psi_U(u)=\psi_U(t^{-1}ut)=\psi_V^{-1}(v)\psi_U(u_0)$. Hence $f_{\mathcal{W}}(\tilde\delta\,(u,1)\,\mathfrak{s}^{\iota}(g)\,h)\,\psi_U^{-1}(u)=f_{\mathcal{W}}(\tilde\delta\,(u_0,1)\,h)\,\psi_U^{-1}(u_0)$, and $u\mapsto u_0$ is a measure preserving bijection of $U_0(\mathbb{A})$, because $U_0$ maps bijectively onto $(\delta^{-1}V^{\diamondsuit}\delta)\backslash U$ and $t$ acts on this quotient. This gives \eqref{L:eq claim}.

Second, for $\text{Re}(s)$ large and $\tilde g_1\in SL_2^{(n)}(\mathbb{A})$,
\begin{equation}\label{L:eq unfolded}
\xi(\varphi^{(n)},f_s)(\tilde g_1)=\int\limits_{G(\mathbb{A})}\int\limits_{U_0(\mathbb{A})}\varphi^{(n)}(\tilde h)\,f_{\mathcal{W}}\big(\tilde\delta\,(u,1)\,\lambda_L(\tilde g_1)\,\lambda_R(\tilde h)\big)\,\psi_U^{-1}(u)\,du\,dh.
\end{equation}
To see this, let $R^\prime=\mathfrak{e}_2(G)\ltimes U$. Since $\lambda_L(\tilde g_1)$ commutes with $\lambda_R(\tilde h)$, the function \eqref{L:eq xi} at $\tilde g_1$ is the integral over $[R^\prime]$ of $\varphi^{(n)}\psi_U^{-1}$ against the Eisenstein series of the right translate of $f_s$ by $\lambda_L(\tilde g_1)$. We unfold it in the same way as in the proof of \cite[Theorem 64]{K}. The double cosets $P(F)\backslash H(F)/R^\prime(F)$ may be grouped modulo $R=\mathfrak{e}_1(G)R^\prime$. The terms outside the open orbit $P\delta R$ vanish by the arguments of \cite[\S 2.3]{C-F-G-K1}, which use cuspidality on the right copy only. The stabilizer $R_\delta$ of $P\delta$ in $R$ is the semidirect product of the group of the elements $t(g,{}^{\iota}g)$ and $\delta^{-1}V^{\diamondsuit}\delta$. It contains $t(g,{}^{\iota}g)$ for every $g\in G$, so $R=R_\delta R^\prime$ and $P\delta R=P\delta R^\prime$, and its intersection with $R^\prime$ is $\delta^{-1}V^{\diamondsuit}\delta$. So the term of the open orbit is an integral over $(\delta^{-1}V^{\diamondsuit}\delta)(F)\backslash R^\prime(\mathbb{A})$, and the integration over $[V]$ produces $f_{\mathcal{W}}$, which gives \eqref{L:eq unfolded}. 

The change of variables $\tilde h\mapsto{}^{\iota}\tilde g_1\tilde h$ in \eqref{L:eq unfolded} replaces $\lambda_L(\tilde g_1)\lambda_R({}^{\iota}\tilde g_1)$ by $\mathfrak{s}^{\iota}(g_1)$ by \eqref{L:eq diag}. Using \eqref{L:eq claim}, we get for $\text{Re}(s)$ large,
\begin{equation}\label{L:eq main full}
\xi(\varphi^{(n)},f_s)(\tilde g_1)=\int\limits_{G(\mathbb{A})}\int\limits_{U_0(\mathbb{A})}\varphi^{(n)}\big({}^{\iota}\tilde g_1\tilde h\big)\,f_{\mathcal{W}}\big(\tilde\delta\,(u,1)\,\lambda_R(\tilde h)\big)\,\psi_U^{-1}(u)\,du\,dh.
\end{equation}
For $\tilde g\in SL_2^{(n)}(\mathbb{A})$ put
\begin{equation}\label{L:eq Z}
Z(f_s)(\tilde g)=\int\limits_{U_0(\mathbb{A})}f_{\mathcal{W}}\big(\tilde\delta\,(u,1)\,\lambda_R(\tilde g)\big)\,\psi_U^{-1}(u)\,du.
\end{equation}
This is a genuine function on $SL_2^{(n)}(\mathbb{A})$. Since $\varphi^{(n)}$ is anti-genuine, the product $Z(f_s)(\tilde h)\,\varphi^{(n)}({}^{\iota}\tilde g_1\tilde h)$ depends only on the image $h$ of $\tilde h$, and \eqref{L:eq main full} reads
\begin{equation}\label{L:eq main}
\xi(\varphi^{(n)},f_s)(\tilde g_1)=\int\limits_{G(\mathbb{A})}Z(f_s)(\tilde h)\,\varphi^{(n)}\big({}^{\iota}\tilde g_1\tilde h\big)\,dh.
\end{equation}

Now let $\varphi^{(n)}$ and $f_s$ be decomposable. By \cite[(3.4)]{K} the functional \eqref{L:eq fW} is the tensor product of a functional on the component over $S$ and of the unique local functionals at the places $\nu\notin S$, normalized to be $1$ on the spherical vectors. Hence $f_{\mathcal{W}}(h)=f_{\mathcal{W},S}(h_S)\prod_{\nu\notin S}f_{\mathcal{W},\nu}(h_\nu)$, where $h=\prod_\nu h_\nu$ is written as a product of elements of the local covers, and $Z(f_s)(\tilde g)=Z_S(\tilde g_S)\prod_{\nu\notin S}Z_\nu(\tilde g_\nu)$, where $Z_S$ and $Z_\nu$ are given by the integral \eqref{L:eq Z} over $U_0(F_S)$ and $U_0(F_\nu)$, with $\tilde\delta_\nu=(\delta,\eta_\nu(\delta)^{-1})$. Since ${}^{\iota}$ is linear, \eqref{L:eq main} gives
\begin{gather}
\xi(\varphi^{(n)},f_s)={}^{\iota}\ell\Big(\xi_S(\varphi_S,f_{s,S})\otimes\bigotimes_{\nu\notin S}\xi_\nu\Big),\notag\\
\xi_S(\varphi_S,f_{s,S})=\int\limits_{G(F_S)}Z_S(\tilde g)\,\overline{\pi}_S(\tilde g)\varphi_S\,dg,\qquad
\xi_\nu=\int\limits_{G(F_\nu)}Z_\nu(\tilde g)\,\overline{\pi}_\nu(\tilde g)\varphi^0_\nu\,dg.\label{L:eq XiS}
\end{gather}

Now we perform the unramified calculation on $\xi_\nu$. Let $\nu\notin S$. First, we show that 
\begin{equation}\label{L:eq Kinv}
Z_\nu\big((r,1)\tilde g\big)=Z_\nu(\tilde g),\qquad r\in K_{G,\nu}.
\end{equation}
In fact, equation \eqref{L:eq diag} gives $\mathfrak{s}^{\iota}_\nu({}^{\iota}r)=\lambda_{L,\nu}(({}^{\iota}r,1))\,\lambda_{R,\nu}((r,1))$, that is
\begin{equation*}
\lambda_{R,\nu}((r,1))=\mathfrak{s}^{\iota}_\nu({}^{\iota}r)\,(\mathfrak{e}_1({}^{\iota}r^{-1}),1).
\end{equation*}
The second factor lies in the left copy, so it commutes with $\lambda_{R,\nu}(\tilde g)$ by Lemma~\ref{L:lem cocycle}, and it lies in the image of $K_{H,\nu}$, so the spherical section is right invariant under it. The integral over $U_0(F_\nu)$ is invariant under the first factor by the local form of \eqref{L:eq claim}. 

For $r\in K_{G,\nu}$, change variables $\tilde g\mapsto(r,1)\tilde g$ in the integral \eqref{L:eq XiS} for $\xi_\nu$ and integrate over $r\in K_{G,\nu}$, taking the measure of $K_{G,\nu}$ to be $1$. By \eqref{L:eq Kinv},
\begin{equation}\label{L:eq average}
\xi_\nu=\int\limits_{G(F_\nu)}Z_\nu(\tilde g)\int\limits_{K_{G,\nu}}\overline{\pi}_\nu\big((r,1)\tilde g\big)\varphi^0_\nu\,dr\,dg.
\end{equation}
The $dr$-integral in \eqref{L:eq average} is equal to $\omega^0_\nu(\tilde g)\varphi^0_\nu$, where $\omega^0_\nu$ is the normalized spherical function attached to $\overline{\pi}_\nu$. This is because $\overline{\pi}_\nu$ is an irreducible constituent of an unramified principal series representation, so every fixed vector of $K_{G,\nu}$ in $\overline{\pi}_\nu$ is a scalar multiple of $\varphi^0_\nu$ \cite[Lemma 6.3]{Mc}. Hence
\begin{equation}\label{L:eq scalar}
\xi_\nu=\Big(\int\limits_{G(F_\nu)}Z_\nu(\tilde g)\,\omega^0_\nu(\tilde g)\,dg\Big)\varphi^0_\nu.
\end{equation}
Make a change of variable $\tilde g\mapsto{}^{\iota}\tilde g$ and expand $Z_\nu$, 
\begin{equation}\label{L:eq K417}
\xi_\nu=\Big(\int\limits_{G(F_\nu)}\int\limits_{U_0(F_\nu)}\omega^0_\nu({}^{\iota}\tilde g)\,f_{\mathcal{W},\nu}\big(\tilde\delta_\nu\,(u,1)\,\lambda_{R,\nu}({}^{\iota}\tilde g)\big)\,\psi_U^{-1}(u)\,du\,dg\Big)\varphi^0_\nu.
\end{equation}
This is the local doubling integral \cite[(4.17)]{K} for $\pi_\nu^{\iota}$ with unramified data. By \cite[Theorem 66]{K} we get
\begin{equation*}
\xi_\nu=\frac{L(ns+\tfrac12,\pi_\nu^{\iota}\times\tau_\nu)}{d_\nu(s)}\,\varphi^0_\nu.
\end{equation*}
Its exterior square factors are $1$ because our $\tau_\nu$ is a representation of a cover of $GL_1$.

Finally $L(s,\pi_\nu^{\iota}\times\tau_\nu)=L(s,\pi_\nu\times\tau_\nu)$. The lift of ${}^{\iota}$ inverts the diagonal torus, that is ${}^{\iota}(a,1)=(a^{-1},1)$. Hence $\pi_\nu^{\iota}$ is the unramified constituent of the principal series representation attached to $\mu^{-1}$, and its Satake parameter is $\text{diag}(t_\nu^{-1},1,t_\nu)$. The Satake parameter is a semisimple conjugacy class in $SO_3(\mathbb{C})$, and $\text{diag}(t_\nu^{-1},1,t_\nu)$ is conjugate to $\text{diag}(t_\nu,1,t_\nu^{-1})$. So $\pi_\nu^{\iota}$ and $\pi_\nu$ have the same Satake parameter $t_{\pi_\nu}$, and $L(s,\pi_\nu^{\iota}\times\tau_\nu)=\det(1-t_{\tau_\nu}t_{\pi_\nu}q_\nu^{-s})^{-1}=L(s,\pi_\nu\times\tau_\nu)$. Hence
\begin{equation*}
\xi_\nu=\frac{L(ns+\tfrac12,\pi_\nu\times\tau_\nu)}{d_\nu(s)}\,\varphi^0_\nu,\qquad\nu\notin S.
\end{equation*}
Since $\pi$ and $\tau$ are unitary, the products $L^S(ns+\tfrac12,\pi\times\tau)$ and $d^S(s)$ converge absolutely for $\text{Re}(s)$ large, so $\bigotimes_{\nu\notin S}\xi_\nu$ is the vector $\varphi^S$ multiplied by $L^S(ns+\tfrac12,\pi\times\tau)/d^S(s)$, and \eqref{L:eq euler} follows from \eqref{L:eq XiS}. In particular $\xi(\varphi^{(n)},f_s)$ lies in the space of the functions ${}^{\iota}\Phi$ with $\Phi$ in the space of $\overline{\pi}$, which realizes $\overline{\pi}^{\iota}$. This completes the proof of Theorem~\ref{L:thm}.

\subsection{The New Way Integral}\label{L:newway}
Recall that the section $f_s$ of Theorem~\ref{Theorem} is constructed in Section~\ref{sec:main theorem} from a section $f^{\prime\prime\prime}_s$ and Schwartz functions $\phi_1,\phi_2\in\mathcal{S}(X(\mathbb{A}))$ with $\phi_1(0,\dots,0,\tfrac12,\tfrac12\xi)=\phi^\prime(\xi)$, and that $f^\prime_s$ is determined by $f^{\prime\prime\prime}_s$ and $\phi_2$. Take $\varphi^{(n)}$, $\phi^\prime$, $f^{\prime\prime\prime}_s$, $\phi_1$ and $\phi_2$ decomposable, with $\varphi_\nu=\varphi^0_\nu$, $f^{\prime\prime\prime}_{s,\nu}$ the normalized spherical section and $\phi^\prime_\nu$, $\phi_{1,\nu}$, $\phi_{2,\nu}$ the characteristic functions of the integral points for $\nu\notin S$, and take the decompositions given by the Dixmier--Malliavin lemma over $F_S$. Then $f_s$ is decomposable and unramified outside $S$. The function \eqref{definition of phi} is the product of the characteristic functions of $X(\mathcal{O}_\nu)$ in $x$ and in $y$ and of $\psi_\nu(-z)$. The operator \eqref{convolved} and the two convolutions of Section~\ref{sec:main theorem} fix the normalized spherical section.

\begin{corollary}\label{L:cor}
For these data and $\text{Re}(s)$ large, the new way integral of Theorem~\ref{Theorem} is equal to
\begin{equation}\label{L:eq cor}
\frac{L^S(ns+\tfrac12,\pi\times\tau)}{d^S(s)}\int\limits_{F\backslash\mathbb{A}}{}^{\iota}\ell\big(\xi_S(\varphi_S,f_{s,S})\otimes\varphi^S\big)\big((x(y),1)\big)\,\psi^{-1}(y)\,dy.
\end{equation}
\end{corollary}

Corollary~\ref{L:cor} does not unfold the new way integral itself. As in \cite[Conjecture 1]{G-S2}, we expect the following.

\begin{conjecture}\label{L:conj}
The new way integral of Theorem~\ref{Theorem} is Eulerian, and represents the partial tensor product $L$-function $L^S(ns+\tfrac12,\pi\times\tau)$.
\end{conjecture}

\section{Cocycle Computations}\label{sec:cocycle}
In this section, we prove Proposition~\ref{prop:cocycle global}. We keep the notation of Section~\ref{sec:FJ descent}: matrices in $Sp_{4n}$ are written in $n\times n$ blocks, $w_0$ is the Weyl element defined there, and $m = \text{diag}(I_n, A, A^\ast, I_n)$, $m^\prime = \text{diag}(A, I_n, I_n, A^\ast)$ for $A \in GL_n$.

Since all entries of $w_0$ and $w_0^{-1}$ are $0$ or $\pm 1$, we have $w_0 \in Sp_{4n}(\mathcal{O}_\nu) = K_\nu$ for every finite $\nu$. A direct blockwise computation gives
\begin{equation}\label{eq:w0conj}
    w_0\, \text{diag}(X_1, X_2, X_3, X_4)\, w_0^{-1} = \text{diag}(X_2, X_4, X_1, X_3).
\end{equation}
In particular $w_0 m w_0^{-1} = m^\prime$. As a block permutation, \eqref{eq:w0conj} says that $w_0$ realizes the permutation $(1\,3\,4\,2)$ on the block positions. Therefore, a positive root $(c,d)$ of $GL_{4n}$, with $c$ in block $i$ and $d$ in block $j$, is made negative by $w_0$ if and only if
\begin{equation}\label{eq:flipped}
    (i,j) \in \{(1,2),\, (1,4),\, (3,4)\}.
\end{equation}

We recall some properties of the block-compatible cocycle on $GL_m(F_\nu)$ constructed in \cite{B-L-S}. Let $T$ denote the diagonal torus, $N$ the subgroup of upper triangular unipotent matrices, and $(\cdot,\cdot)_\nu$ the $n$-th order Hilbert symbol of $F_\nu$. Let $\mathfrak{M}$ denote the set of Weyl representatives introduced in \cite{B-L-S}: for a simple root $\alpha = (i, i+1)$, set
\begin{equation*}
    w_\alpha = \text{diag}\Bigl(I_{i-1}, \begin{pmatrix} & -1 \\ 1 & \end{pmatrix}, I_{m-i-1}\Bigr).
\end{equation*}
For an element $w$ of the Weyl group with reduced expression $w = s_{\alpha_1}\cdots s_{\alpha_{\ell(w)}}$, let $\eta_w = w_{\alpha_1}\cdots w_{\alpha_{\ell(w)}}$. This product is independent of the choice of the reduced expression. Define the length function $\ell(\eta_w) = \ell(w)$, and the set $\mathfrak{M} = \{\eta_w\}_{w}$. Every monomial matrix $g$ factors uniquely as $g = \mathbf{t}(g)\,\eta$, where $\mathbf{t}(g) \in T$ and $\eta \in \mathfrak{M}$. If the nonzero entries of $g$ are $\pm 1$, then the entries of $\mathbf{t}(g)$ are $\pm 1$. 

We list four properties of the block compatible local cocycle $\sigma_\nu$ \cite[\S 3]{B-L-S}. For $g_1, g_2 \in GL_m(F_\nu)$, $n_1, n_2, n \in N$, $t, t_1, t_2\in T$ and $\eta, \eta_1, \eta_2 \in \mathfrak{M}$:
\begin{enumerate}[label=(P\arabic*)]
    \item $\sigma_\nu(n_1g_1,\, g_2n_2) = \sigma_\nu(g_1, g_2)$;
    \item $\sigma_\nu(g_1 n,\, g_2) = \sigma_\nu(g_1,\, ng_2)$;
    \item $\sigma_\nu(\eta,t) = \prod\limits_{\substack{a = (i,j),\ i<j \\ wa <0}}(-t_j, t_i)_\nu$, the product over the positive roots $a = (i,j)$ made negative by the underlying permutation $w$ of $\eta$;
    \item $\sigma_\nu(t_1\eta_1,\, t_2\eta_2) = \sigma_\nu(t_1,\, {}^{\eta_1}t_2)\,\sigma_\nu(\eta_1, t_2)$, provided $\ell(\eta_1\eta_2) = \ell(\eta_1) + \ell(\eta_2)$.
\end{enumerate}
We also use the formula $\sigma_\nu(t, t^\prime) = \prod_{i<j}(t_i, t^\prime_j)_\nu$ for $t, t^\prime \in T$ \cite[\S 3, Lemma 1]{B-L-S}. The Hilbert symbol is bimultiplicative and kills $n$-th powers since we assume that $\mu_{2n} \subset F^\times$ and $-1 = \zeta_{2n}^{\,n} \in (F^\times)^n$. Therefore, $(x, \pm 1)_\nu = (\pm 1, x)_\nu = 1$ for all $x \in F_\nu^\times$. Hence, for all $t, t^\prime \in T$ and all $t_0, t_0^\prime \in T$ with entries $\pm 1$,
\begin{equation}\label{eq:pm1}
    \sigma_\nu(tt_0,\, t^\prime) = \sigma_\nu(t,\, t^\prime t_0^\prime) = \sigma_\nu(t, t^\prime), \qquad \sigma_\nu(t, t_0) = \sigma_\nu(t_0, t) = 1.
\end{equation}
By (P3) and bi-multiplicativity of $\sigma_\nu$, for $\eta \in \mathfrak{M}$,
\begin{equation}\label{eq:pm1 weyl}
    \sigma_\nu(\eta,\, t t_0) = \sigma_\nu(\eta, t), \qquad \sigma_\nu(\eta, t_0) = 1.
\end{equation}

We also record an elementary fact about lengths.

\begin{lemma}\label{lem:length}
Let $w_1, w_2$ be permutations of $\{1, \dots, 4n\}$, and consider the partition of $\{1, \dots, 4n\}$ into four consecutive blocks of size $n$. Suppose that one of $w_1, w_2$ preserves each block, and the other maps blocks onto blocks, preserving the order inside each block. Then $\ell(w_1w_2) = \ell(w_1) + \ell(w_2)$.
\end{lemma}
\begin{proof}
The length of a permutation is the number of its inversions. Denote by $u$ the permutation preserving each block and by $v$ the other one. Consider a pair $c < d$ in $\{1, \dots, 4n\}$. If $c$ and $d$ lie in the same block, then $v$ neither creates nor destroys an inversion, because it preserves the order inside blocks. So the pair is inverted by $w_1w_2$ if and only if it is inverted by $u$. If $c$ and $d$ lie in different blocks, then $u$ does not move them across blocks. Therefore, the pair is inverted by $w_1w_2$ if and only if the two blocks are inverted by $v$. Hence the number of inversions of $w_1w_2$ is $\ell(u) + \ell(v)$, in either order of composition.
\end{proof}

\begin{lemma}\label{lem:cocycle local}
For every place $\nu$ of $F$ and every $A \in GL_n(F_\nu)$,
\begin{equation*}
    \sigma_\nu(m^\prime, w_0) = 1 \qquad \text{and} \qquad \sigma_\nu(w_0, m) = 1.
\end{equation*}
\end{lemma}
\begin{proof}
Write the Bruhat decomposition $A = n_1 t w n_2$ over $F_\nu$, where $n_1, n_2 \in N_n(F_\nu)$, $t \in T_n(F_\nu)$, and $w$ is a permutation matrix. Since $X \mapsto X^\ast$ is a homomorphism, $A^\ast = n_1^\ast\, t^\ast\, w^\ast\, n_2^\ast$, where $n_1^\ast, n_2^\ast \in N_n(F_\nu)$, $t^\ast = J_n t^{-1} J_n \in T_n(F_\nu)$ (its $l$-th entry is $t_{n+1-l}^{-1}$), and $w^\ast = J_n w J_n$ is a permutation matrix.

For the first cocycle, factor blockwise
\begin{equation*}
    m^\prime = \text{diag}(n_1, I, I, n_1^\ast) \cdot \underbrace{\text{diag}(t, I, I, t^\ast)}_{=:\, \widehat{t}\, \in\, T_{4n}} \cdot \underbrace{\text{diag}(w, I, I, w^\ast)}_{=:\, \widehat{w}} \cdot\, \text{diag}(n_2, I, I, n_2^\ast),
\end{equation*}
where the first and last factors lie in $N_{4n}$. Then
\begin{align*}
    \sigma_\nu(m^\prime, w_0)
    &= \sigma_\nu\bigl(\widehat{t}\,\widehat{w}\:\text{diag}(n_2, I, I, n_2^\ast),\ w_0\bigr) && \text{by (P1)}\\
    &= \sigma_\nu\bigl(\widehat{t}\,\widehat{w},\ \text{diag}(n_2, I, I, n_2^\ast)\, w_0\bigr) && \text{by (P2)}\\
    &= \sigma_\nu\bigl(\widehat{t}\,\widehat{w},\ w_0\, \text{diag}(I, n_2, n_2^\ast, I)\bigr) && \text{by \eqref{eq:w0conj}}\\
    &= \sigma_\nu\bigl(\widehat{t}\,\widehat{w},\ w_0\bigr). && \text{by (P1)}
\end{align*}
Write $\widehat{w} = t_0\,\eta$ and $w_0 = \mathbf{t}(w_0)\,\eta_1$ with $\eta, \eta_1 \in \mathfrak{M}$. Since the nonzero entries of $\widehat{w}$ and $w_0$ are $\pm 1$, the entries of $t_0$ and $\mathbf{t}(w_0)$ are $\pm 1$. The two permutation matrices $\widehat{w}$ and $w_0$ satisfy the condition of the above lemma, hence $\ell(\eta\,\eta_1) = \ell(\eta) + \ell(\eta_1)$ by Lemma~\ref{lem:length}, and (P4) applies:
\begin{equation*}
    \sigma_\nu\bigl(\widehat{t}\,\widehat{w},\ w_0\bigr) = \sigma_\nu\bigl((\widehat{t}\,t_0)\,\eta,\ \mathbf{t}(w_0)\,\eta_1\bigr) = \sigma_\nu\bigl(\widehat{t}\,t_0,\ {}^{\eta}\mathbf{t}(w_0)\bigr)\,\sigma_\nu\bigl(\eta,\ \mathbf{t}(w_0)\bigr) = 1,
\end{equation*}
where the last equality is by \eqref{eq:pm1} and \eqref{eq:pm1 weyl}.

For the second cocycle, factor blockwise
\begin{equation*}
    m = \text{diag}(I, n_1, n_1^\ast, I) \cdot \underbrace{\text{diag}(I, t, t^\ast, I)}_{=:\, \widehat{t}^{\,\prime}} \cdot \underbrace{\text{diag}(I, w, w^\ast, I)}_{=:\, \widehat{w}^{\,\prime}} \cdot\, \text{diag}(I, n_2, n_2^\ast, I).
\end{equation*}
Then
\begin{align*}
    \sigma_\nu(w_0, m)
    &= \sigma_\nu\bigl(w_0,\ \text{diag}(I, n_1, n_1^\ast, I)\, \widehat{t}^{\,\prime}\widehat{w}^{\,\prime}\bigr) && \text{by (P1)}\\
    &= \sigma_\nu\bigl(w_0\, \text{diag}(I, n_1, n_1^\ast, I),\ \widehat{t}^{\,\prime}\widehat{w}^{\,\prime}\bigr) && \text{by (P2)}\\
    &= \sigma_\nu\bigl(\text{diag}(n_1, I, I, n_1^\ast)\, w_0,\ \widehat{t}^{\,\prime}\widehat{w}^{\,\prime}\bigr) && \text{by \eqref{eq:w0conj}}\\
    &= \sigma_\nu\bigl(w_0,\ \widehat{t}^{\,\prime}\widehat{w}^{\,\prime}\bigr). && \text{by (P1)}
\end{align*}
Write $w_0 = \mathbf{t}(w_0)\,\eta_1$ as before, and $\widehat{w}^{\,\prime} = t_0^\prime\,\eta^\prime$ with $\eta^\prime \in \mathfrak{M}$ and $t_0^\prime \in T$ with entries $\pm 1$. The underlying permutation of $\widehat{w}^{\,\prime}$ preserves each block, so $\ell(\eta_1\,\eta^\prime) = \ell(\eta_1) + \ell(\eta^\prime)$ by Lemma~\ref{lem:length}, and by (P4), \eqref{eq:pm1} and \eqref{eq:pm1 weyl},
\begin{equation*}
    \sigma_\nu\bigl(w_0,\ \widehat{t}^{\,\prime}\widehat{w}^{\,\prime}\bigr) = \sigma_\nu\bigl(\mathbf{t}(w_0)\,\eta_1,\ (\widehat{t}^{\,\prime}t_0^\prime)\,\eta^\prime\bigr) = \sigma_\nu\bigl(\mathbf{t}(w_0),\ {}^{\eta_1}(\widehat{t}^{\,\prime}t_0^\prime)\bigr)\,\sigma_\nu\bigl(\eta_1,\ \widehat{t}^{\,\prime}t_0^\prime\bigr) = \sigma_\nu\bigl(\eta_1,\ \widehat{t}^{\,\prime}\bigr).
\end{equation*}
Finally, we evaluate $\sigma_\nu(\eta_1, \widehat{t}^{\,\prime})$ with (P3). The coordinates of $\widehat{t}^{\,\prime} = \text{diag}(I, t, t^\ast, I)$ are $1$ on blocks $1$ and $4$, $t_l$ on block $2$, and $t^\ast_l = t^{-1}_{n+1-l}$ on block $3$. By \eqref{eq:flipped}, the roots flipped by $w_0$ pair the blocks $(1,2)$, $(1,4)$ and $(3,4)$, so
\begin{equation*}
    \sigma_\nu(\eta_1, \widehat{t}^{\,\prime})
    = \underbrace{\prod_{l}\,(-t_l,\, 1)_\nu^{\,n}}_{(1,2)}
    \cdot
    \underbrace{(-1,1)_\nu^{\,n^2}}_{(1,4)}
    \cdot
    \underbrace{\prod_{l}\,(-1,\, t^\ast_l)_\nu^{\,n}}_{(3,4)}
    = 1,
\end{equation*}
where the first two groups of symbols are trivial by bimultiplicativity, and the third because $(-1, x)_\nu = 1$.
\end{proof}

\begin{proof}[Proof of Proposition~\ref{prop:cocycle global}]
By the formula for conjugation in Section~\ref{sec:prelim}, the conjugation multiplies $\epsilon$ by $c = \rho(w_0, m)\,\rho(m^\prime, w_0)^{-1} = \prod_\nu c_\nu$, where, by the definition of $\rho_\nu$ and since $w_0 m_\nu = m^\prime_\nu w_0$,
\begin{align*}
    c_\nu = \rho_\nu(w_0, m_\nu)\,\rho_\nu(m^\prime_\nu, w_0)^{-1}
    &= \frac{\eta_\nu(w_0)\,\eta_\nu(m_\nu)}{\eta_\nu(w_0 m_\nu)}\cdot\frac{\eta_\nu(m^\prime_\nu w_0)}{\eta_\nu(m^\prime_\nu)\,\eta_\nu(w_0)}\cdot\frac{\sigma_\nu(w_0, m_\nu)}{\sigma_\nu(m^\prime_\nu, w_0)}\\
    &= \frac{\eta_\nu(m_\nu)}{\eta_\nu(m^\prime_\nu)}\cdot\frac{\sigma_\nu(w_0, m_\nu)}{\sigma_\nu(m^\prime_\nu, w_0)}.
\end{align*}
We show $c_\nu = 1$ for every $\nu$.

\emph{Case 1: $\nu \in S_0$.} Here $\rho_\nu = \sigma_\nu$, and both values are $1$ by Lemma~\ref{lem:cocycle local}, so $c_\nu = 1$.

\emph{Case 2: $\nu \notin S_0$ and $m_\nu \in K_\nu$.} Since $w_0 \in K_\nu$, we also have $m^\prime_\nu = w_0 m_\nu w_0^{-1} \in K_\nu$. Both pairs $(w_0, m_\nu)$ and $(m^\prime_\nu, w_0)$ lie in $K_\nu \times K_\nu$, where $\rho_\nu$ is trivial. Therefore, $c_\nu = 1$.

\emph{Case 3: $\nu \notin S_0$ and $m_\nu \notin K_\nu$.} Then $m^\prime_\nu \notin K_\nu$ (otherwise $m_\nu = w_0^{-1} m^\prime_\nu w_0 \in K_\nu$). By the definition of $\eta_\nu$, $\eta_\nu(m_\nu) = \eta_\nu(m^\prime_\nu) = 1$, so $c_\nu = \sigma_\nu(w_0, m_\nu)\,\sigma_\nu(m^\prime_\nu, w_0)^{-1} = 1$ by Lemma~\ref{lem:cocycle local}.
\end{proof}

\section{Proof of Lemma~\ref{lem:addacolumn}}\label{sec:addacolumn}
In this section, we prove Lemma~\ref{lem:addacolumn}. Although $n$ is assumed to be odd throughout the paper, the argument does not depend on the parity of $n$, and we give the proof for both even and odd $n$.

\begin{proof}[Proof of Lemma~\ref{lem:addacolumn}]
The idea of the proof is to apply Fourier expansion, over $(F\backslash\mathbb{A})^n$, to the new $(n+1)$-st column in the second integral of \eqref{addacolumn}. The constant term of the expansion is exactly the first integral of \eqref{addacolumn}. We claim that the other terms contain the Fourier coefficients corresponding to the orbit $(1^{n-1}, n+1)$ of $GL_{2n}$. Hence, they are zero because $f_s$ is a section induced from $\Delta(\Theta_{GL^{(n)}_n}, 2)|\cdot|^s$, which has orbit $(n^2)$ that is incomparable with $(1^{n-1}, n+1)$. By conjugating the integrand with a permutation matrix, we only need to show that the following integral is a Fourier coefficient corresponding to the orbit $\mathcal{O} =(1^{n-1}, n+1)$:
\begin{equation} \label{lemma3int}
\int\limits f_s\left( \iota \begin{pmatrix}
1&&&&x_{1,n+1}&x_{1,n+2}&\cdots&x_{1,2n}\\
&1&&&x_{2,n+1}&x_{2,n+2}&\cdots&x_{2,2n}\\
&&\ddots&&\vdots&\vdots&\ddots&\vdots\\
&&&1&x_{n,n+1}&x_{n,n+2}&\cdots&x_{n,2n}\\
&&&&1&x_{n+1,n+2}&\cdots&x_{n+1,2n}\\
&&&&&1&\ddots&\vdots\\
&&&&&&\ddots&x_{2n-1,2n}\\
&&&&&&&1
\end{pmatrix}h\right)\psi(\sum\limits_{i=n}^{2n-1}x_{i, i+1}) dx.
\end{equation}
Here, as in \eqref{addacolumn}, every variable is integrated over $F\backslash\mathbb{A}$.
 
\textbf{Case: n is even.} For the definition of Fourier coefficients associated with nilpotent orbits, we follow the convention in \cite[\S 2]{G} except that we use a different but equivalent torus up to conjugation by a permutation matrix. In particular,  consider the one dimensional torus
\begin{equation*}
     h(t) = \text{diag}(t^{n}, t^{n-2}, ..., t^{2}, 1, t^{-2}, .., t^{-n}, 1, ..., 1).
\end{equation*}
For a root $\alpha = (i,j)$ of $GL_{2n}$ and $r \in \mathbb{A}$, write $x_{\alpha}(r) = I_{2n} + r\,e_{i,j}$, so that $\{x_\alpha(r)\}$ is the corresponding one-parameter root subgroup. For each positive integer $i$, let $U_i$ be the subgroup of $GL_{2n}$ generated by the root subgroups $x_{\alpha}(r)$ that satisfy 
\begin{equation*}
    h(t)x_\alpha(r)h(t)^{-1} = x_\alpha(t^i r).
\end{equation*}
Let $V$ denote the subgroup of $GL_{2n}$ generated by $U_i$ for $i\geq 2$. Let $\pi$ denote an automorphic representation of the group $GL_{2n}(\mathbb{A})$. Let $\varphi$ denote a vector in the space of $\pi$. The Fourier coefficient associated with the orbit $\mathcal{O} =(1^{n-1}, n+1)$ is defined as
\begin{equation}\label{FCdefinition}
    \int\limits_{V(F)\backslash V(\mathbb{A})} \varphi(vg) \psi_V(v) dv,
\end{equation}
where $\psi_V$ is a generic character on $V(F)\backslash V(\mathbb{A})$ associated to $\mathcal{O}$. More concretely, $V$ is generated by root subgroups $I_{2n} + x_{i,j} e_{i,j}$, for
\begin{gather*}
    1\leq i \leq n+1,\ i < j \leq n+1, \text{ or}\\
    1\leq i \leq \tfrac{n}{2},\ n+2 \leq j \leq 2n, \text{ or}\\
    n+2 \leq i \leq 2n,\ \tfrac{n}{2}+2 \leq j \leq n+1.
\end{gather*}
The character can be chosen to be 
\begin{equation*}
    \psi_V(v) = \psi\left(\sum\limits_{i=1}^{n} v_{i, i+1}\right),
\end{equation*}
where $v_{i,i+1}$ denotes the $(i, i+1)$ entry of $v$ and $\psi$ is a nontrivial character on $F\backslash\mathbb{A}$.

We make a root exchange between the lower triangular unipotent subgroups generated by roots $(i,j)$ for $n+2 \leq i \leq 2n,\ \tfrac{n}{2}+2 \leq j \leq n+1$ with the upper triangular unipotent subgroups generated by roots $(i,j)$ for $\tfrac{n}{2}+1 \leq i \leq n,\ n+2 \leq j \leq 2n$. Then for a certain choice of vector $\varphi^\prime$ in the space of $\pi$, the Fourier coefficient associated with the orbit $\mathcal{O}$ is of the form
\begin{equation*}
    \int \varphi^{\prime}\left(\begin{pmatrix}
1      & x_{12} & x_{13} & \cdots & x_{1n}    & x_{1,n+1} & x_{1,n+2} & \cdots & x_{1,2n}\\
       & 1      & x_{23} & \cdots & x_{2n}    & x_{2,n+1} & x_{2,n+2} & \cdots & x_{2,2n}\\
       &        & \ddots & \ddots & \vdots    & \vdots    & \vdots    &        & \vdots\\
       &        &        & \ddots & x_{n-1,n} & \vdots    & \vdots    &        & \vdots\\
       &        &        &        & 1         & x_{n,n+1} & x_{n,n+2} & \cdots & x_{n,2n}\\
       &        &        &        &           & 1         &           &        & \\
       &        &        &        &           &           & 1         &        & \\
       &        &        &        &           &           &           & \ddots & \\
       &        &        &        &           &           &           &        & 1
\end{pmatrix} g \right)\psi(\sum_{i=1}^{n}x_{i, i+1}) dx.
\end{equation*}
By conjugating the matrix in the integrand of \eqref{lemma3int} by $J_{2n}$, we see that \eqref{lemma3int} is a Fourier coefficient associated with the orbit $(1^{n-1}, n+1)$ when $n$ is even. 

\textbf{Case: n is odd.}
In this case, we consider the one dimensional torus
\begin{equation*}
     h^\prime(t) = \text{diag}(t^{n}, t^{n-2}, ..., t^{1}, t^{-1}, .., t^{-n}, 1, ..., 1).
\end{equation*}
With $V$ and $\psi_V$ defined in a similar way with $h^\prime(t)$, the Fourier coefficient associated with the orbit $\mathcal{O}$ is of the form \eqref{FCdefinition}, where $V$ is generated by root subgroups $I_{2n} + x_{i,j} e_{i,j}$ for
\begin{gather*}
    1\leq i \leq n+1,\ i < j \leq n+1, \text{ or}\\
    1\leq i \leq \tfrac{n-1}{2},\ n+2 \leq j \leq 2n, \text{ or}\\
    n+2 \leq i \leq 2n,\ \tfrac{n+1}{2}+2 \leq j \leq n+1.
\end{gather*} 
The character can be chosen to be 
\begin{equation*}
    \psi_V(v) = \psi\left(\sum\limits_{i=1}^{n} v_{i, i+1}\right),
\end{equation*}
where $v_{i,i+1}$ denotes the $(i, i+1)$ entry of $v$ and $\psi$ is a nontrivial character on $F\backslash\mathbb{A}$.
By making a root exchange between the lower triangular unipotent subgroups generated by roots $(i,j)$ for $n+2 \leq i \leq 2n,\ \tfrac{n+1}{2}+2 \leq j \leq n+1$ with the upper triangular unipotent subgroups generated by roots $(i,j)$ for $\tfrac{n+1}{2}+1 \leq i \leq n,\ n+2 \leq j \leq 2n$, we obtain an integral of the form 
\begin{equation*}
    \int \varphi^{\prime}\left(\begin{pmatrix}
1 & x_{12} & \cdots & x_{1,(n+1)/2} & \cdots & x_{1,n+1} & x_{1,n+2} & \cdots & x_{1,2n}\\
  & 1      & \cdots & x_{2,(n+1)/2} & \cdots & x_{2,n+1} & x_{2,n+2} & \cdots & x_{2,2n}\\
  &        & \ddots & \vdots        &        & \vdots    & \vdots    &        & \vdots\\
  &        &        & 1             & \cdots & x_{(n+1)/2,\,n+1} & 0 & \cdots & 0\\
  &        &        &               & \ddots & \vdots    & \vdots    &        & \vdots\\
  &        &        &               &        & 1         & x_{n,n+2} & \cdots & x_{n,2n}\\
  &        &        &               &        &           & 1         &        & \\
  &        &        &               &        &           &           & \ddots & \\
  &        &        &               &        &           &           &        & 1
\end{pmatrix} g \right)\psi(\sum_{i=1}^{n}x_{i, i+1}) dx,
\end{equation*}
where the integral is over the upper triangular entries in the first $n$ rows except for row $(n+1)/2$, in which the integral omits columns $n+2, ..., 2n$.

By conjugating the matrix in the integrand of \eqref{lemma3int} by $J_{2n}$, we see that \eqref{lemma3int} contains a Fourier coefficient associated with the orbit $(1^{n-1}, n+1)$ as an inner integral when $n$ is odd. Therefore, the lemma follows.
\end{proof}

\section*{Statement on AI usage}
The author acknowledges the use of AI during the preparation of this manuscript as follows. In Section~\ref{L:sec}, the embeddings of the two copies of $SL_2^{(n)}$ and the map $\mathfrak{s}^{\iota}$ were suggested by AI. In addition, AI models were used for proofreading and for suggesting improvements of the presentation.

\section*{Acknowledgements}
I am very grateful to my advisor Solomon Friedberg for providing a tremendous amount of advice, support and encouragement. I would like to thank David Ginzburg for proposing this project and for helpful discussions on technical details of his paper with David Soudry. This research was supported by the NSF, grant number DMS-2401309 (PI Solomon Friedberg).


\newpage
\begin{thebibliography}{AAAAA}
\bibitem[B-F]{B-F} Bump, Daniel; Friedberg, Solomon. \emph{Metaplectic generating functions and Shimura integrals.} Automorphic forms, automorphic representations, and arithmetic (Fort Worth, TX, 1996), 1--17, Proc. Sympos. Pure Math., 66, Part 2, Amer. Math. Soc., Providence, RI, 1999.

\bibitem[B-F-G]{B-F-G} Bump, Daniel; Furusawa, Masaaki; Ginzburg, David. \emph{Nonunique models in the Rankin--Selberg method.} J. Reine Angew. Math. 468 (1995), 77--111.

\bibitem[B-L-S]{B-L-S} Banks, W.; Levy, J.; Sepanski, M. R. \emph{Block-compatible metaplectic cocycles.} J. Reine Angew.
Math. 507 (1999), 131--163.

\bibitem[C-F-G-K1]{C-F-G-K1} Yuanqing Cai, Solomon Friedberg, David Ginzburg, Eyal Kaplan. \emph{Doubling Constructions and Tensor Product L-Functions: the linear case.}
Inventiones Mathematicæ 217(3) (2019), 985--1068.

\bibitem[C-F-G-K2]{C-F-G-K2} Yuanqing Cai, Solomon Friedberg, David Ginzburg, Eyal Kaplan. \emph{Doubling Constructions for Covering Groups and Tensor Product L-Functions.}
In arXiv:1601.08240.

\bibitem[C-F-K]{C-F-K} Yuanqing Cai, Solomon Friedberg, Eyal Kaplan. \emph{Doubling constructions: Global functoriality for non-generic cuspidal representations.}
Annals of Mathematics 200(3) (2024), 893--966.

\bibitem[D-M]{D-M} Dixmier, Jacques; Malliavin, Paul.
\emph{Factorisations de fonctions et de vecteurs indefiniment differentiables.}
Bull. Sci. Math. (2) 102 (1978), no. 4, 307--330.

\bibitem[F-G]{F-G} Friedberg, Solomon; Ginzburg, David. \emph{Classical theta lifts for higher metaplectic covering groups.} Geom. Funct. Anal. 30 (2020), no. 6, 1531--1582.

\bibitem[G]{G} Ginzburg, David. \emph{Certain conjectures relating unipotent orbits to automorphic representations.} Israel J. Math. 151(1) (2006), 323--355.

\bibitem[G2]{G2} Ginzburg, David. \emph{Generating functions on covering groups.} Compos. Math. 154 (2018), no. 4, 671--684.

\bibitem[G-R-S1]{G-R-S1} Ginzburg, David; Rallis, Stephen; Soudry, David. \emph{L-functions for symplectic groups.} Bull. Soc. Math. France 126 (1998), no. 2, 181--244.

\bibitem[G-R-S2]{G-R-S2}  Ginzburg, David; Rallis, Stephen; Soudry, David. \emph{The descent map from automorphic representations of GL(n) to classical groups.} World Scientific Publishing Co. Pte. Ltd., Hackensack, NJ, 2011. x+339 pp. ISBN: 978-981-4304-98-6; 981-4304-98-0.

\bibitem[G-S1]{G-S1} Ginzburg, David; Soudry, David. \emph{Two Identities relating Eisenstein series on classical groups.}
arXiv:1808.01572. 

\bibitem[G-S2]{G-S2} Ginzburg, David; Soudry, David. \emph{Integrals derived from the doubling method.}
Int. Math. Res. Not. IMRN 2020, no. 24, 10553--10596.

\bibitem[I]{I}   Ikeda, Tamotsu. \emph{On the theory of Jacobi forms and Fourier-Jacobi coefficients of Eisenstein series.} J. Math. Kyoto Univ. 34 (1994), no. 3, 615--636.

\bibitem[J-L]{J-L} Jiang, Dihua; Liu, Baiying. \emph{On Fourier coefficients of automorphic forms of $GL(n)$.} Internat. Math. Res. Notices 2013(17) (2013), 4029--4071.

\bibitem[J-Y]{J-Y} Jin, Yubo; Yan, Pan. \emph{$L$-functions for $Sp(2n)\times GL(k)$ via non-unique models.} Int. Math. Res. Not. IMRN 2025, no. 5, rnaf035.

\bibitem[K]{K} Kaplan, Eyal. \emph{Doubling Constructions and Tensor Product L-Functions: Coverings of the Symplectic Group.} 
Forum of Mathematics, Sigma 13 (2025). https://doi.org/10.1017/fms.2024.63. 

\bibitem[K-P]{K-P} D.~A.~Kazhdan and S.~J.~Patterson,
\emph{Metaplectic forms},
Publ. Math. Inst. Hautes \'Etudes Sci. \textbf{59} (1984), 35--142.

\bibitem[M]{M} Moore, Calvin C. \emph{Group extensions of $p$-adic and adelic linear groups.} Inst. Hautes \'Etudes Sci. Publ. Math. 35 (1968), 157--222.

\bibitem[Mc]{Mc} McNamara, Peter J. \emph{Principal series representations of metaplectic groups over local fields.} In Multiple Dirichlet series, L-functions and automorphic forms, Progr. Math. 300, Birkh\"auser/Springer, New York, 2012, 299--327.

\bibitem[PS-R1]{PS-R1}  Piatetski-Shapiro, Ilya; Rallis, Stephen. ``$L$ functions for the classical groups.'' \emph{Explicit constructions of automorphic L-functions.} Lecture Notes in Mathematics, 1254. Springer-Verlag, Berlin, 1987. vi+152 pp. ISBN: 3-540-17848-1

\bibitem[PS-R2]{PS-R2} Piatetski-Shapiro, Ilya; Rallis, Stephen. \emph{A new way to get Euler products.}
J. Reine Angew. Math. 392 (1988), 110--124.

\bibitem[P-S1]{P-S1} Pollack, Aaron; Shah, Shrenik. \emph{On the Rankin--Selberg integral of Kohnen and Skoruppa.} Math. Res. Lett. 24 (2017), no. 1, 173--222.

\bibitem[P-S2]{P-S2} Pollack, Aaron; Shah, Shrenik. \emph{The spin $L$-function on $GSp_6$ via a non-unique model.} Amer. J. Math. 140 (2018), no. 3, 753--788.

\bibitem[T]{T} Takeda, Shuichiro. \emph{Metaplectic tensor products for automorphic representation of $\widetilde{GL}(r)$.} Canad. J. Math. 68 (2016), 179--240.


\bibitem[Y]{Y} Yan, Pan. \emph{$L$-function for $Sp(4)\times GL(2)$ via a non-unique model.} Acta Math. Sin. (Engl. Ser.) (2026). https://doi.org/10.1007/s10114-026-4579-x

\end{thebibliography}
\end{document}